\documentclass[review,10pt]{elsarticle}
\usepackage{lipsum}
\usepackage{amsthm}
\makeatletter
\def\ps@pprintTitle{%
 \let\@oddhead\@empty
 \let\@evenhead\@empty
 \def\@oddfoot{}%
 \let\@evenfoot\@oddfoot}
\makeatother
\usepackage[margin=2.5cm]{geometry}
\usepackage{lineno,hyperref}
\usepackage{amsmath,amssymb,latexsym, amscd}
\usepackage{exscale, eps fig, graphics}
\usepackage{array}
\usepackage{xcolor}
\usepackage{float}
\restylefloat{table}
\newcommand\numberthis{\addtocounter{equation}{1}\tag{\theequation}}
\usepackage{algorithmic}
\usepackage{placeins}
\usepackage{longtable}
\usepackage{makecell}
\usepackage[utf8]{inputenc}
\usepackage[justification=centering]{caption}
\usepackage{graphicx}
\usepackage{caption}
\usepackage{subcaption}
\usepackage{mathtools}
\hypersetup{pdfauthor={Name}}
\newtheorem{thm}{Theorem}[section]
\newtheorem{lem}[thm]{Lemma}

\newtheorem{defn}[thm]{Definition}

\newcommand{\BigO}[1]{\ensuremath{\operatorname{O}\left(#1\right)}}
\newcommand{\BigOk}[1]{\ensuremath{\operatorname{O}_k\left(#1\right)}}

\newcommand{\BigOke}[1]{\ensuremath{\operatorname{O}_{k,\epsilon}\left(#1\right)}}

\renewcommand{\geq}{\geqslant}
\renewcommand{\leq}{\leqslant}

\renewcommand{\r}{\rangle}
\newcommand{\BigOkb}[1]{\ensuremath{\operatorname{O}_{k,\mathcal{B}}\left(#1\right)}}
\newcommand{\BigOkbl}[1]{\ensuremath{\operatorname{O}_{k,\mathcal{B},l}\left(#1\right)}}
\newcommand{\BigOhh}[1]{\ensuremath{\operatorname{O}_{h_1,h_2,\cdots,h_m}\left(#1\right)}}
\newcommand{\BigOhhe}[1]{\ensuremath{\operatorname{O}_{h_1,h_2,\cdots,h_m,\epsilon}\left(#1\right)}}

\begin{document}

\begin{frontmatter}

\title{On the distribution of index of Farey Sequences}

\author[mymainaddress]{Bittu}
\ead{bittui@iiitd.ac.in}

\author[mymainaddress]{Sneha Chaubey}
\ead{sneha@iiitd.ac.in}

\author[mymainaddress]{Shivani Goel}
\ead{shivanig@iiitd.ac.in}

\address[mymainaddress]{Department of Mathematics, IIIT Delhi, New Delhi 110020}

\begin{abstract}
In this article, we study the distribution of index of Farey fractions which was first introduced and studied by Hall and Shiu. We provide asymptotic formulas for moments of index of Farey fractions twisted by Dirichlet characters for Farey fractions with $\mathcal{B}$- free denominators. 
%For the square-free case, it is done in \cite{MR2424917}, with an error in their proof. In an appendix, rectify their error and provide new formulas for moments of index with square-free denominators. 
%Moreover, we also study higher level correlations of the index function generalizing earlier known results on two level correlations. 
Additionally, we reconsider the squarefree case earlier done in \cite{MR2424917}, and obtain new results for moments of indices with square-free denominators. We also study higher level correlations of the index function generalizing earlier known results on two level correlations.
\end{abstract}


\end{frontmatter}

%\linenumbers
\section{Introduction and main results}
\subsection{Index of Farey fractions}
An ascending sequence of fractions in the unit interval $(0,1]$ is called a Farey sequence of order $Q$ if for any fraction $\gamma_i=a/q$,  $\gcd(a,q)=1$ and $0<a\le q\le Q$. Denote $\mathcal{F}_Q$ as the Farey % $\mathcal{F}_Q=\{\gamma_1,\gamma_2,...,\gamma_N \} $ denotes the Farey
sequence of order $Q$ with $1/Q= \gamma_1  <\gamma_2<...<\gamma_N =1$.
Let $\gamma_{i-1}=\dfrac{a_{i-1}}{q_{i-1}}<\gamma_{i}=\dfrac{a_{i}}{q_{i}}<\gamma_{i+1}=\dfrac{a_{i+1}}{q_{i+1}}$ be three consecutive Farey fractions, then the ratio
  \[\nu_Q(\gamma_i)=\frac{q_{i-1}+q_{i+1}}{q_i}=\frac{a_{i-1}+a_{i+1}}{a_i}\numberthis\label{index}\] is an integer and is
 called the index of the Farey fraction $\gamma_i$ in $\mathcal{F}_Q.$ In particular, we take  $\nu_Q(\gamma_1)=1$, and $\nu_Q(\gamma_{N_Q})=2Q$.
 %For example, the indices of the fractions in  $\mathcal{F}_5$ are
 %\begin{center}
%\begin{tabular}{|c| c| c| c| c| c| c| c| c| c| c|}
%\hline
%$\gamma_i$  & $\frac{1}{5}$ & $\frac{1}{4}$ &$\frac{1}{3}$ &$\frac{2}{5}$&$\frac{1}{2}$ &$\frac{3}{5}$ &$\frac{2}{3}$ &$\frac{3}{4}$ &$\frac{4}{5}$ &${1}$ \\[1ex]\hline
%$\nu_Q(\gamma_i)$ &  $1$ & $2$&  $3$&  $1$& $5$& $1$& $3$& $2$& $1$& $10$\\[0.5ex]\hline
%\end{tabular}
%\end{center}
The index satisfies the following property:
\[\left\lfloor{\frac{2Q+1}{q_i}}\right\rfloor-1\le \nu(\gamma_i)\le \left\lfloor{\frac{2Q}{q_i}}\right\rfloor\] which implies that for a given Farey fraction, the corresponding index function can only take one of the two values $\left\lfloor{\dfrac{2Q}{q_i}}\right\rfloor$ or $\left\lfloor{\dfrac{2Q}{q_i}}\right\rfloor-1$. Using this property, Hall and Shiu \cite{Shiu} proved closed form formulas and asymptotic formulas for the first and second moments of the index function given by
%defined a new notion of Farey deficiency denoted by $\delta(q_i)$. Deficiency is the number of fractions $\gamma_i\in \mathcal{F}_Q$  with denominator $q_i$ such that $\nu(\gamma_i)=\left\lfloor{\dfrac{2Q}{q_i}}\right\rfloor-1$. They proved that
\[\sum_i\nu_Q(\gamma_i)=3N(Q)-1,\]and
%\[\sum_{q_i=1}^Q\delta(q_i)=Q(2Q+1)-N_{2Q}-2N_{Q}+1.\]
%They also showed the following asymptotic formula 
\[\sum_i\nu_Q(\gamma_i)^2=\frac{24}{\pi^2}Q^2\left(\log2Q-\frac{\zeta'(2)}{\zeta(2)}-\frac{17}{8}+2\gamma\right)+\BigO{Q\log^2Q},\]
where $\gamma$ and $\zeta$ are the Euler's constant and the Riemann zeta function, respectively.
The counting function of Farey fractions with index  $\nu(\gamma_i)=\left\lfloor{\dfrac{2Q}{q_i}}\right\rfloor-1$ is called deficiency of $\gamma_i$ and is denoted by $\delta(\gamma_i)$. A formula for average of $\delta(\gamma_i)$ is proved 
in \cite{Shiu} as 
\[\sum_{q_i=1}^Q\delta(q_i)=Q(2Q+1)-N_{2Q}-2N_{Q}+1.\]
There has been a considerable interest in the study of the index function mainly upon imposing extra divisibility constraints on the Farey denominators.
Partial sums of Farey indices weighted according to the parity of the Farey denominators was studied by Hall in \cite{Hall1}. In \cite{MR2342669}, the authors study the moments of the index function for Farey fractions with denominators in a fixed residue class, and in \cite{MR2424917} with square free denominators in a fixed residue class. 
In this article, we extend this line of investigation and study distribution of the index function by evaluating the moments of Farey indices with $\mathcal{B}$-free Farey denominators in an arithmetic progression. Similar divisibility constraint with $\mathcal{B}$-free Farey denominators in an arithmetic progression was imposed by the authors in \cite{discre} where they obtain absolute bounds on discrepancy of Farey fractions with such denominators.
The notion of $\mathcal{B}$-free numbers was introduced by Erd\H{o}s in \cite{MR0207673}, as a generalization of square free numbers. 
\begin{defn}
Given an infinite set of integers $\mathcal{B}=\{b_1, b_2, \cdots\}$ of integers greater than $1$, we say that a positive integer is $\mathcal{B}$-free if it is not divisible by any element of $\mathcal{B}$. \end{defn}
Denote the set of all primes by $\mathcal{P}$.
For positive integer $k\ge 2$, $\mathcal{B}=\{p^k: p\in \mathcal{P}\}$ gives rise to $k$-free integers.
%For a set $\mathcal{B}$ of positive numbers $1<b_1,b_2,\cdots$ such that \[\sum_{k=1}^{\infty}\frac{1}{b_k}<\infty \ \text{with}\ \gcd(b_i,b_j)=1\ \text{for}\ i\ne j,\] $n$ is a $\mathcal{B}$-free number, if $n$ is not divisible by $b_k$ for all $b_k\in \mathcal{B}$.
It is known \cite{MR0207673} that for all large enough $N$ and some $0<c<1$, the interval $[N,N+N^c]$ contains at least one $\mathcal{B}$-free number. For a survey of results and open problems on the distribution of $\mathcal{B}$-free numbers one may refer to the well-known articles of \cite{MR2248751} and \cite{Avdeeva}. 

    The $l$-th moment of Farey indices with $\mathcal{B}$-free Farey denominator in an arithmetic progression is given by
    \begin{align}\label{moment}
    \mathcal{M}_{l,\mathcal{B}}(u,k, Q):=\sum_{\substack{\gamma_i=\frac{b}{s}\in \mathcal{F}_Q\\s\equiv u\pmod k\\ s \ \text{is} \  \mathcal{B} \text{-free}}}\nu(\gamma_i)^l.
    \end{align}
    In here, we provide asymptotic formulas for the above sum for first, second, and higher moments for the set
    \begin{align} \label{setB}
    \mathcal{B}=\{p\in\mathcal{P}:  \sum_{p\in\mathcal{B}}\dfrac{1}{{p_i}^\sigma}<\infty \  \text{for some } \sigma<\theta,  \text{ where }  1/2<\theta<1\}. \end{align}
    \begin{thm}\label{theorem3}
For fixed positive integers $k, u$ such that $\gcd(k,u)=1$, and set $\mathcal{B}$ defined in \eqref{setB},
for all large positive integers $Q$, we have 
    \[\mathcal{M}_{1,\mathcal{B}}(u,k,Q)=
       \dfrac{3Q^2}{ 2\phi(k)\zeta(2)}\prod_{\substack{p\in \mathcal{B}\\p\nmid k}} \left(1+\dfrac{1}{p} \right)^{-1}\prod_{p|k}\left(1+\dfrac{1}{p} \right)^{-1}+\BigOkb
{Q^{1+\theta}(\log Q)^{3/2}}.\]
\end{thm}
\begin{thm}\label{thm4}
For fixed positive integers $k, u$ such that $\gcd(k,u)=1$, and $\mathcal{B}$ is a set of primes defined in \eqref{setB}, we have
\begin{align*}
\mathcal{M}_{2,\mathcal{B}}(u,k,Q)&=\dfrac{4Q^2}{\phi(k)\zeta(2)}\left[\log 2Q+2\gamma -\dfrac{\zeta^\prime(2)}{\zeta(2)}-\dfrac{17}{8} +\sum_{\substack{p\in\mathcal{B}\\p\nmid k}}\dfrac{p\log p}{p^2-1}+\sum_{p|k}\dfrac{p\log p}{p^2-1}\right] \prod_{\substack{p\in\mathcal{B}\\p\nmid k}}\left(1+\dfrac{1}{p}\right)^{-1}\\ &\times \prod_{p|k}\left(1+\dfrac{1}{p}\right)^{-1}+\frac{4Q^2}{\phi(k)}\sum_{\chi\ne\chi_0}\frac{\chi(\Bar{u})L(1,\chi)}{L(2,\chi)}\prod_{p\in \mathcal{B}}\left(1-\dfrac{\chi(p)}{p}\right)\left(1-\dfrac{\chi(p)}{p^{2}}\right)^{-1}
\\&+\BigOkb{Q^{1+\theta}(\log Q)^2}.\end{align*}
    \end{thm}
\begin{thm}\label{thm5}
For fixed positive integers $k, u$ such that $\gcd(k,u)=1$, and set $\mathcal{B}$ defined in \eqref{setB}, then for $l\ge3$, we have
    \begin{align*}
    \mathcal{M}_{l,\mathcal{B}}(\chi,Q)= A(Q,l)+E(Q,l),
\end{align*}
where \[A(Q,l)=\frac{2^2Q^l}{\phi(k)}\sum_{\chi\ne\chi_0}\frac{\chi(\Bar{u})L(l-1,\chi)}{L(l,\chi)}\prod_{p\in \mathcal{B}}\left(\sum_{n=0}^{\infty}\frac{\chi(p)^n}{(p^{l-1})^n} \right)^{-1}\left(\sum_{n=0}^{\infty}\frac{\chi(p)^n\mu(n)}{(p^{l})^n} \right)^{-1},\]
and
\[E(Q,l)=\left\{\begin{array}{cc}
   \BigOkb{Q^2\log Q},  & \mbox{if} \ l=3,\\
  \BigOkb{Q^{l-1}}   & \mbox{if}\ l\ge4.
\end{array}\right.\]
    \end{thm}
Note that the set $\mathcal{B}=\{p^2: p\in\mathcal{P}\}$ is different from $\mathcal{B}$ defined in \eqref{setB}, and thus our results do not generalize the results for squarefree Farey denominators in \cite{MR2424917}.  We noticed that there are errors in the square free case done in  \cite{MR2424917}, as a result of which the statements of their main theorems change. We take a note of this gap and provide results for the case of squarefree denominators in arithmetic progressions in an appendix.
\subsection{Correlations of Farey indices}
A different but related direction in the study of distribution of Farey indices is via their correlation functions. For this, we extend the sequence $\mathcal{F}_Q$ by identifying $\gamma_{i+N(Q)}=\gamma_i+1$ for all $i\in\mathbb{Z}$. For any $m\ge 1$, the $m$-correlations function for Farey index for Farey fractions over the full unit interval and a sub-interval is defined in the following way.
\begin{defn}\label{def1}
 For every positive integers $h_1,h_2,...,h_m$, $m\ge 1$, the  $m$-correlation function for Farey indices is given by
 \[S_{h_1,h_2\cdots,h_m}(Q):= \sum_{\gamma_i\in\mathcal{F}_Q} \nu_Q(\gamma_i) \nu_Q(\gamma_{i+h_1})\cdots\nu_Q(\gamma_{i+h_m}).\]
 \end{defn}
\begin{defn}\label{def2}
 For every positive integers $h_1,h_2,...,h_m$, $m\ge 1$, the  $m$-correlation function for Farey indices for Farey fractions belonging to a sub-interval of $[0,1]$ is given by
 \[S_{h_1,h_2\cdots,h_m,t}(Q):= \sum_{\gamma_i\in\mathcal{F}_Q\cap(0,t]} \nu_Q(\gamma_i) \nu_Q(\gamma_{i+h_1})\cdots\nu_Q(\gamma_{i+h_m}).\]
 \end{defn}
For second correlation function ($m=2$), it was conjectured in \cite{Shiu}, that for every positive integer $h$, there is a positive constant $A(h)$
such that 
\[S_h(Q)= \sum_{\gamma_i\in\mathcal{F}_Q} \nu_Q(\gamma_i) \nu_Q(\gamma_{i+h}) \sim A(h)N(Q)\]
as $Q\to \infty.$ This was later proved in \cite{Gologan} with asymptotic formulas for second correlations for Farey indices for both full and sub-intervals. In our next main results, we derive asymptotic formulas for $m$-correlation functions for Farey indices for any $m\ge 2$.  
\begin{thm}\label{Theoremone}
   For every $h_1,h_2,...,h_m\ge 1$, we have
   \[S_{h_1,h_2,...,h_m}(Q) = A(h_1,h_2,...,h_m)N(Q) + \BigOhh{Q\log^2Q}. \]
   as $Q\to\infty,$ where $A(h_1,h_2,...,h_m)$ is a positive rational constant given in terms of Farey transformations. 
  \end{thm}
\begin{thm}\label{Theoremtwo}
  For every $h_1,h_2,...,h_m\ge 1$ and every $t\in [0,1]$, we have for all $\epsilon >0$,
 \[S_{h_1,h_2,...,h_m,t}(Q)=tA(h_1,h_2,...,h_m) N(Q) + \BigOhhe{Q^{3/2+\epsilon}\log Q}, \]
 where $A(h_1,h_2,...,h_m)$ is same as in the above theorem.
 \end{thm}
We remark that it would be interesting to obtain explicit bounds as well as mean values of constants $A(h)$, and $A(h_1,h_2,...,h_m)$, which now is stated in terms of Farey transformations.
 \subsection{Organization}
 The article is organized as follows. 
Section 2 contains preliminary results required to proof the main Theorems. In Section 3, we prove Theorem \ref{theorem3}. Section 4 contains a result on the average of deficiency and proof of Theorem \ref{thm4}. The proof for higher moments is accomplished in Section 5. Section 6 covers proofs of Theorems \ref{Theoremone} and \ref{Theoremtwo}. We devote Section 7 for a discussion on the constant $A(h_1,h_2,...,h_m)$. The suarefree case is worked out in the appendix in Section 8.  %\cite{MR2424917}. 
     \subsection{Notation} 
     \begin{enumerate}
         \item  %Mobius, unit function $\epsilon(s)$, Euler phi.
         We use $\phi$ to denote the Euler totient function,  unit function $\epsilon(s)$, $\mu$ the Mobius function.
         \item $\zeta(s)$ is the Riemann zeta function, and $L(s,\chi)$ is the Dirichlet L-function. Here $\chi$ is Dirichlet character.
          \item $\chi_0$ denotes the principal Dirichlet character.
         \item  We use the Vinogradov $\ll$ asymptotic notation, and the big oh O(·) asymptotic notation. Dependence on a parameter is denoted by a subscript.
        \item $\tau(n)$ is the number of divisors of n.
         \item For a real number x, $\lfloor{x}\rfloor$ denotes the integer part of x.
         \item $(f*g)(n)=\sum_{d|n}f(d)g(n/d)$ is the Dirichlet convolution of the arithmetic functions $f$ and $g$.
     \end{enumerate}
     \subsection{Acknowledgements}
     The authors would like to thank Tomos Parry for valuable inputs during the preparation of this article. The first author acknowledges support by the Science and Engineering Research Board, Department of Science and Technology, Government of India under grant SB/S2/RJN-053/2018.
     \section{Preliminaries and general setup}
     In this section, we review some results required to prove our main results. 
        \\ For fixed positive integers $k, u$ with $\gcd(k,u)=1$, and $ u\Bar{u} \equiv 1 \pmod{k}$, in view of the identity
 \begin{displaymath}
\frac{1}{\phi(k)}\sum_{\chi}\chi(\Bar{u}s)=\left
\{\begin{array}{lr}
  1   & \text{if} \ s \equiv u \pmod{k}, \\
  0   & \text{otherwise},
\end{array}
\right.
\end{displaymath}     
     the $l$-th moment $\eqref{moment}$ can be written as
\begin{align*}
    \mathcal{M}_{l,\mathcal{B}}(u,k,Q)&=\frac{1}{\phi(k)}\sum_{\substack{\gamma_i=\frac{b}{s}\in \mathcal{F}_Q\\\mu_{\mathcal{B}}(s)=1}}\nu(\gamma_i)^l\sum_{\chi}\chi(\Bar{u}s)=\frac{1}{\phi(k)}\sum_{\chi}\chi(\Bar{u})\sum_{\substack{\gamma_i=\frac{b}{s}\in \mathcal{F}_Q\\\mu_{\mathcal{B}}(s)=1}}\chi(s)\nu(\gamma_i)^l
    \\
    &=\frac{1}{\phi(k)}\sum_{\chi}\chi(\Bar{u})\mathcal{M}_{l,\mathcal{B}}(\chi,Q), %\numberthis\label{moment1}
\end{align*}
where
\begin{align}\label{moment2}
\mathcal{M}_{l,\mathcal{B}}(\chi, Q):=\sum_{\substack{\gamma_i=\frac{b}{s}\in \mathcal{F}\\\mu_{\mathcal{B}}(s)=1}}\chi(s)\nu(\gamma_i)^l=\sum_{\substack{s\le Q \\ \mu_{\mathcal{B}}(s)=1}}\chi(s)\sum_{\gamma_i=\frac{b}{s}\in \mathcal{F}}\nu(\gamma_i)^l.\end{align}
We express the partial sum involving index separately for $l=1,2$ and $l\ge 3$. \\
(i) For $l=1$, we use the definition of index \eqref{index}, and write
\[\sum_{\gamma_i=\frac{b}{s}\in \mathcal{F}_Q}\nu(\gamma_i)=\frac{2}{s}\sum_{\substack{r=Q-s+1 \\ \gcd{(r,s)=1}}}^{Q} r=2\sum_{d|s}\mu(d) \left\lfloor{\frac{Q}{d}}\right\rfloor-\phi(s)+\epsilon(s).\numberthis\label{41} \]
This leads to 
\begin{align*}
    \mathcal{M}_{1,\mathcal{B}}(\chi,Q)=2\sum_{\substack{s\leq Q\\ \mu_{\mathcal{B}}(s)=1}}\chi(s)\sum_{d|s}\mu(d)\left\lfloor{\frac{Q}{d}}\right\rfloor - \sum_{\substack{s\leq Q\\\mu_{\mathcal{B}}(s)=1}}\chi(s)\phi(s) +1. \numberthis\label{511}
\end{align*}
(ii) For $l=2$, we write the mean value of square of the index using deficiency. 
%Recall, that the deficiency $\delta(q_i)$ is the number of Farey fractions $\gamma_i\in \mathcal{F}_Q$  with denominator $q_i$ such that $\nu(\gamma_i)=\left\lfloor{\dfrac{2Q}{q_i}}\right\rfloor-1$. 
%Let $\gamma_i=b/s\in\mathcal{F}_Q$, 
Recall, that the index $\nu(\gamma_{i})$ can take at most two values $\lfloor{2Q/q_i}\rfloor$ or $\lfloor{2Q/q_i}\rfloor-1$, and the deficiency $\delta(q_i)$ is the number of Farey fractions $\gamma_i\in \mathcal{F}_Q$ with denominator $q_i$ such that $\nu(\gamma_i)=\left\lfloor{\dfrac{2Q}{q_i}}\right\rfloor-1$. Moreover, it is well known that there are $\phi(s)$ fractions in $\mathcal{F}_Q$ with denominator s. This yields
\begin{align*}
    T(s)&:=\sum_{\gamma_i=\frac{b}{s}\in \mathcal{F}_Q}\nu(\gamma_i)=(\phi(s)-\delta(s))\left\lfloor{\frac{2Q}{s}} \right\rfloor+\delta(s)\left(\left\lfloor{\frac{2Q}{s}}\right\rfloor-1 \right)\\&=\phi(s)\left\lfloor{\frac{2Q}{s}}\right\rfloor-\delta(s),\numberthis\label{42}\end{align*}
    and
 \begin{align*}
    \sum_{\gamma_i=\frac{b}{s}\in \mathcal{F}_Q}\nu(\gamma_i)^2 &=(\phi(s)-\delta(s)){\left \lfloor {\frac{2Q}{s}}\right \rfloor}^2+\delta(s)\left(\left \lfloor {\frac{2Q}{s}}\right \rfloor -1\right)^2 \\&=\phi(s){\left \lfloor {\frac{2Q}{s}}\right \rfloor}^2-\delta(s)\left(2{\left \lfloor {\frac{2Q}{s}}\right \rfloor}-1 \right). \numberthis\label{b1}
\end{align*}
Note that from $\eqref{41}$ and $\eqref{42}$, we get an alternate expression for $\delta(s)$ given by
\begin{align*}
    \delta(s)=\phi(s)\left(\left\lfloor{\frac{2Q}{s}}\right\rfloor +1\right)-2\sum_{d|s}\mu(d)\left\lfloor{\frac{Q}{d}}\right\rfloor-\epsilon(s). \numberthis\label{second2}
\end{align*}
Consequently from \eqref{b1}, we have
\begin{align*}
    \mathcal{M}_{2,\mathcal{B}}(\chi,Q)&=\sum_{\substack{s\leq Q  \\ \mu_{\mathcal{B}}(s)=1}}\chi(s)\left(\phi(s){\left \lfloor {\frac{2Q}{s}}\right \rfloor}^2-\delta(s)\left(2{\left \lfloor {\frac{2Q}{s}}\right \rfloor}-1 \right) \right) \\&=X_{\chi}(Q)-2Y_{\chi}(Q)+\sum_{\substack{s\leq Q \\ \mu_{\mathcal{B}}(s)=1}}\chi(s)\delta(s) ,\numberthis\label{76}
\end{align*}
where
\[X_{\chi}(Q):=\sum_{\substack{s\leq Q \\ \mu_{\mathcal{B}}(s)=1}}\chi(s)\phi(s){\left \lfloor {\frac{2Q}{s}}\right \rfloor}^2, \ 
Y_{\chi}(Q):=\sum_{\substack{s\leq Q \\ \mu_{\mathcal{B}}(s)=1}}\chi(s)\delta(s){\left \lfloor {\frac{2Q}{s}}\right \rfloor}.\]
(iii) For $l\ge 3$, we use the fact that for any three consecutive Farey fractions $\gamma_{i-1}=\frac{a_{i-1}}{q_{i-1}}<\gamma_{i}=\frac{a_{i}}{q_{i}}<\gamma_{i+1}=\frac{a_{i+1}}{q_{i+1}},$ their denominators are related (see \cite{Hall}) as 
 \[ q_{i+1} =\left \lfloor {\frac{Q+q_{i-1}}{q_{i}}}\right \rfloor q_{i} - q_{i-1}.\numberthis\label{relation1}\] This together with definition \eqref{index} enable us to write
 \begin{align*}
     \sum_{\gamma_i=\frac{b}{s}\in \mathcal{F}_Q}\nu(\gamma_i)^l&=\sum_{\substack{Q-s<r<Q\\(r,s)=1}}\left\lfloor{\frac{Q+r}{s}}\right\rfloor^l=\sum_{\substack{Q-s<r<Q\\(r,s)=1}}\left(\left(\frac{Q+r}{s} \right)^l+\BigO{\left(\frac{Q+r}{s} \right)^{l-1}}\right)\\
    &=\sum_{\substack{Q-s<r<Q\\(r,s)=1}}\left(\frac{Q+r}{s} \right)^l+\BigO{\sum_{\substack{s\leq Q\\\mu_{\mathcal{B}}(s)=1}}\sum_{\substack{Q-s<r<Q\\(r,s)=1}}\left(\frac{Q+r}{s} \right)^{l-1}}\\
    &=\sum_{\substack{Q-s<r<Q\\(r,s)=1}}\left(\left(\frac{2Q}{s} \right)^l-\left(\frac{2Q}{s} \right)^l+\left(\frac{Q+r}{s} \right)^l\right)+\BigO{Q^{l-1}}.
 \end{align*}
As a result, substituting in \eqref{moment2}, we obtain
    \begin{align} \label{highmoment}
    \mathcal{M}_{l,\mathcal{B}}(\chi,Q)=\sum_{\substack{s\leq Q\\\mu_{\mathcal{B}}(s)=1}}\chi(s)\sum_{\substack{Q-s<r<Q\\(r,s)=1}}\left(\left(\frac{2Q}{s} \right)^l-\left(\frac{2Q}{s} \right)^l+\left(\frac{Q+r}{s} \right)^l\right)+\BigO{Q^{l-1}}.
\end{align}
The estimates for \eqref{511}, \eqref{76}, and \eqref{highmoment} will be handled separately in later sections. The principle tool we use for the first and second moments is the Perron's formula. 
 \begin{lem}\cite[Theorem 2, p. 132]{MR1342300}\label{parronlemma}
 Let $F(s):=\sum_{n=1}^{\infty}f(n) {n}^{-s}$ be the Dirichlet series for the arithmetic function $f(n)$, with abscissa of convergence $\sigma_a$. If $\alpha>\max(0,\sigma_a)$, $T\ge 1$ and $x\ge 1$, then \[\sum_{n\le x}f(n)=\frac{1}{2\pi i}\int_{\alpha-iT}^{\alpha+iT}F(s)\frac{x^s}{s}ds+R(T),\] where \[R(T)\ll \frac{x^{\alpha}}{T}\sum_{n=1}^{\infty}\frac{|f(n)|}{n^{\alpha}|\log(x/n)|}.\] 
 \end{lem}
 \begin{lem}
 \label{zetaintegral}
 For a complex number $s=\sigma+it$, with $1/2<\sigma<1$, we have
 \[\int_{0}^{T}\frac{\zeta(\sigma+it)}{|\sigma+it|}\ll \log T. \]
\end{lem}
\begin{proof}
We write
 \begin{align*}
    \int_{0}^{T}\frac{|\zeta(\sigma+i t)|}{|\sigma+ i t|}dt &\leq \left(\int_{0}^{1}+\sum_{n=0}^{\log T}\int_{2^n}^{2^{n+1}} \right)\frac{|\zeta(\sigma+i t)|}{|\sigma+ i t|}dt \ll 1+ \log T  \smash{\displaystyle\max_{0 \leq n \leq \lfloor{\log T}\rfloor}}\int_{2^n}^{2^{n+1}}\frac{|\zeta(\sigma+i t)|}{|\sigma+ i t|}dt \\
    &\ll 1+ \log T  \smash{\displaystyle\max_{0 \leq n \leq \lfloor{\log T}\rfloor}}\frac{1}{2^n}\int_{0}^{2^{n+1}}\left|\zeta\left(\sigma+i t\right)\right|dt.\numberthis\label{a3}
\end{align*}
Applying the Cauchy-Schwarz inequality, we have
\begin{align*}
    \int_{0}^{2^{n+1}}|\zeta\left(\sigma+i t\right)|dt &\leq \left(\int_{0}^{2^{n+1}}1^2dt \right)^{\frac{1}{2}} \left(\int_{0}^{2^{n+1}}\left|\zeta\left(\sigma+i t\right)\right|^2dt \right)^{\frac{1}{2}} 
     \ll 2^{n+1}, \numberthis\label{a4}
\end{align*}
here we also use the mean value formula (\cite[Theorem 1.11, p. 28]{MR792089}) which gives 
\[\int_{0}^{2^{n+1}}|\zeta\left(\sigma+i t\right)|^2dt\sim \zeta(2\sigma)2^{n+1}.\]
Substituting \eqref{a4} in \eqref{a3}, gives the required result. 
\end{proof}
We also use the following standard bounds of $\zeta(s)$ \cite[page 47]{MR882550},
\begin{displaymath}
 \zeta(\sigma+it) \ll \left\{
   \begin{array}{lr}
     {t^{\frac{1- \sigma}{2}} \log t}, &\  0\le \sigma \le 1,\\
     { \log t},  &\ 1\le\sigma \le 2,\\
      {1} , &\ \sigma \ge 2.
     \end{array}
    \right.\numberthis\label{zetabound}
\end{displaymath}
For non principal characters $\chi\pmod{k}$, we use the following bounds of $L(s,\chi)$ (see \cite{MR551704})
\begin{displaymath}
 L(\sigma+it,\chi) \ll_k \left\{
   \begin{array}{lr}
     t^{1/2-2\sigma}\log t , &\  0\le \sigma \le 1/2,\\
    t^{-\sigma}\log t , &\  1/2\le \sigma \le 1,\\
     { \log t},  &\ 1\le\sigma \le 2,\\
     {1} , &\ \sigma \ge 2.
     \end{array}
    \right.\numberthis\label{lbound}
\end{displaymath}
\subsection{Preliminaries for correlation results}
In here, we summarize some results required for proofs of Theorems \ref{Theoremone} and Theorem \ref{Theoremtwo}. Part of this can also be found in \cite{Gologan}. 
The index of a Farey fraction can be written in terms of a transformation on Farey triangle. To do this, we will employ elementary properties of the well known bijective and area preserving transformation $T$ of Farey triangle \cite{Boca}  
  $\mathcal{T} = \{(x,y)\in [0,1]^2; x+y>1 \}, $
 defined by
 \[ T(x,y) = \left(y,\left \lfloor{\frac{1+x}{y}}\right \rfloor y-x\right).\numberthis\label{tfunction} \]
 The Farey triangle $\mathcal{T}$ can be written as a disjoint union of sets $\mathcal{T}_k$, 
 $\mathcal{T} = \cup_{k=1}^{\infty}\mathcal{T}_{k},$
 where
 \[\mathcal{T}_k =\left \{(x,y)\in \mathcal{T} ;\left \lfloor {\frac{1+x}{y}}\right \rfloor = k \right\}, \hspace{8mm} k\in\mathbb{N}.\numberthis\label{taufunction}\]
Therefore, for $(x,y)\in \mathcal{T}_{k}$, $T(x,y)=(y,ky-x)$.
 Furthermore, for all integer $i\geq 0$, we estimate iterates  $T^i$ of ${T}$ given by
 \[T^i(x,y) = (L_i(x,y),L_{i+1}(x,y)), \]%\numberthis\label{functiont}
 where
 \[L_{i+1}(x,y)=\kappa_i(x,y)L_i(x,y)-L_{i-1}(x,y),\hspace{2mm}   L_0(x,y)=x,\hspace{2mm}   L_1(x,y)=y,\]
 and \[ \kappa_i(x,y)=\kappa_{i-1} \circ T(x,y), \hspace{2mm} \kappa_1(x,y)=\left\lfloor{\frac{(1+x)}{y}}\right \rfloor.\]
An important connection with Farey fractions \cite{Boca} is given by 
 \[T\left(\frac{q_{i-1}}{Q},\frac{q_i}{Q}\right)=\left(\frac{q_i}{Q},\frac{q_{i+1}}{Q}\right).\]
 %where $q_{i-1},q_i$ and $q_{i+1}$ are denominators of three consecutive Farey fractions from $\mathcal{F}_Q.$ 
 It implies that
 \[\kappa_1 \left(\frac{q_{i-1}}{Q},\frac{q_i}{Q}\right)=\left\lfloor{\frac{Q+q_{i-1}}{q_i}}\right\rfloor, \numberthis\label{kapparesult}\]
 and for $s\in \mathbb{N},$
\begin{align*} 
 \kappa_{s+1} \left(\frac{q_{i-1}}{Q},\frac{q_i}{Q}\right) &=\kappa_1\circ T^s\left(\frac{q_{i-1}}{Q},\frac{q_i}{Q}\right)=\kappa_1 \left(\frac{q_{i+s-1}}{Q},\frac{q_{i+s}}{Q}\right)=\left\lfloor{\frac{Q+q_{i+s-1}}{q_{i+s}}}\right\rfloor.\numberthis\label{kappa1}
 \end{align*}
%  We also note that
%  \[\mathcal{T}_k =\left \{(x,y)\in \mathcal{T} ;\ \kappa_1(x,y) = k \right\}.\numberthis\label{tauk}\]
  Using \eqref{kapparesult} in conjunction with \eqref{relation1}, one obtains
 \[\nu_Q(\gamma_i)=\left\lfloor{\frac{Q+q_{i-1}}{q_{i}}}\right\rfloor = \kappa_1\left(\frac{q_{i-1}}{Q},\frac{q_i}{Q}\right).\numberthis\label{nukappa}\]
 % \textcolor{blue}{The following lemma is proved using \cite[Lemma 2.1]{Gologan} by taking $g(a,b)$ to be 1}.
 \begin{lem}\label{lemma1}\cite[Corollary 2.2]{Gologan}
 Let $\Omega\subset [0,R_1]\times[0,R_2]$ be a region in $\mathbb{R}^2$ with rectifiable boundary $\partial\Omega$, and let $R\ge \min\{R_1,R_2\}$. Then we have,
 \[\#(\Omega\cap \mathbb{Z}_{vis}^2)=\frac{6\ \text{area}(\Omega)}{\pi^2}+\BigO{R+\text{length}(\partial\Omega)\log R+\frac{\text{area}(\Omega)}{R}}.\]
 \end{lem}
Note that for a bounded region $\Omega$ and for a linear transformation $\Phi$ on $\mathbb{R}^2$ defined as $\Phi(x,y)=(ax+by,cx+dy)$, where 
$ \begin{pmatrix}
 a & b\\
 c & d
 \end{pmatrix}\in SL_2(\mathbb{Z})$, 
 %\textcolor{blue}{the elements of $\Omega\cap \mathbb{Z}_{vis}^2$ are in one to one correspondence with $\Phi\Omega\cap \mathbb{Z}_{vis}^2$, since the matrix that defines $\Phi$ is unimodular,}  
 since $\Phi$ is unimodular, we have 
$\#(\Omega\cap \mathbb{Z}_{vis}^2)=\#(\Phi\Omega\cap \mathbb{Z}_{vis}^2).$
  So, when $\Omega\subseteq\mathcal{T}$\[\#(Q\Omega\cap \mathbb{Z}_{vis}^2)=\#(Q(T\Omega)\cap \mathbb{Z}_{vis}^2)=\cdots=\#(Q(T^h\Omega)\cap \mathbb{Z}_{vis}^2)\numberthis\label{note1}.\]
 \begin{lem}\label{Corollaryone}
Let $h = \max (h_1, h_2,\cdots, h_m)$ and $\mathcal{T}^*_{k} = \bigcup_{n=k}^{\infty} \mathcal{T}_n$. Assume that $ \min (l_i,l_j)>c_h = 2^{h+1}, i \ne j, 1\leq i, j \leq {m+1}$ and  $ l_{m+1} = k.$ Then
 \[\mathcal{T}^*_k \cap T^{-h_1}\mathcal{T}^*_{l_1}\cap T^{-h_2}\mathcal{T}^*_{l_2}\cap \cdots \cap T^{-h_m}\mathcal{T}^*_{l_m} = \emptyset, \]
 %where 
% \[h = \max (h_1, h_2,\cdots, h_m),\] and  \[\mathcal{T}^*_{k} = \bigcup_{n=k}^{\infty} \mathcal{T}_n. \]
 \end{lem}
 \begin{proof}
  Proof is similar to \cite[Corollary 2.5]{Gologan}.
 \end{proof}
We remark that using \cite[Lemma 3.4]{MR2021008}, the value of $c_h$ in Lemma \ref{Corollaryone} can be lowered upto $4h+2$. Also, by \cite[Lemma 3.4]{MR2021008}, we have for all $k\geq 5$
 \[T^{\pm1}\mathcal{T}^{*}_{k}\subset \mathcal{T}_1, \numberthis\label{note4}\] and for $h\ge 2$ and all $l\ge c_h=4h+2$, we have $\cup_{i=2}^hT^{\pm i}\mathcal{T}^{*}_{k}\subset \mathcal{T}_2.$ In particular, we obtain \[\mathcal{T}^*_k \cap T^{-h_1}\mathcal{T}^*_{l_1}\cap T^{-h_2}\mathcal{T}^*_{l_2}\cap \cdots \cap T^{-h_m}\mathcal{T}^*_{l_m} = \emptyset,\numberthis\label{note3}\]where $k\ge 3$ and $l_i\ge c_h$ for any $1\le i\le m$. 
 Some other properties of Farey fractions used in the note are summarized in the lemma below. 
\begin{lem}\label{subinterval}\cite[section 8]{Boca}
For two neighboring Farey fractions $p_i/q_i < p_{i+1}/q_{i+1}$ of order $ Q$ in an interval $I$, the relation $p_{i+1}q_i - p_i q_{i+1} = 1$ implies that $p_i\equiv -\overline{q_{i+1}} \pmod {q_i}$. The notation $\bar{x}\pmod n$ is used for the multiplicative inverse of $x\pmod n$ in the interval $[1, n]$ for positive integers $x$ and $n$ with $\gcd(x, n) = 1$. Hence, the conditions $p_i/q_i, p_{i+1}/q_{i+1}\in I=(x,y]$ are equivalent to $\overline{q_{i+1}}\in [(1-y)q_i,(1-x)q_i)$, and $\overline{q_i}\in (q_{i+1}x, q_{i+1}y]$.
\end{lem}
\section{First Moment}
%In this section we prove Theorem \ref{theorem3}. 
\subsection{Proof of Theorem \ref{theorem3}}
Using \eqref{511}, the first moment can be expressed as
\begin{align*}
    \mathcal{M}_{1,\mathcal{B}}(\chi,Q)=2\sum_{\substack{s\leq Q\\ \mu_{\mathcal{B}}(s)=1}}\chi(s)\sum_{d|s}\mu(d)\left\lfloor{\frac{Q}{d}}\right\rfloor - \sum_{\substack{s\leq Q\\\mu_{\mathcal{B}}(s)=1}}\chi(s)\phi(s) +1=2S_1-S_2+1. \numberthis\label{51}
\end{align*}
We evaluate the first sum on the right side of \eqref{51} as
\begin{align*}
  S_1&=Q\sum_{d\leq Q}\frac{\mu(d)}{d}\sum_{\substack{s\leq Q \\ d|s \\ \mu_{\mathcal{B}}(s)=1}}\chi(s) +\BigO{\sum_{s\leq Q}\tau(s)}= Q\sum_{d\leq Q}\frac{\mu(d)}{d}\sum_{\substack{s\leq Q \\ d|s }}\chi(s) \mu_{\mathcal{B}}(s) +\BigO{Q \text{log}Q} \\
  &= Q\sum_{d\leq Q}\frac{\mu(d)\chi(d) \mu_{\mathcal{B}}(d)}{d}\sum_{m\leq \frac{Q}{d}}\chi(m) \mu_{\mathcal{B}}(m) +\BigO{Q \text{log}Q}. \numberthis\label{52}
\end{align*}
In order to estimate the inner sum, we use Lemma \ref{parronlemma}, where we put $x=\frac{Q}{d}+\frac{1}{2}$,
\[S_{11}:=\sum_{m\leq \frac{Q}{d}}\chi(m) \mu_{\mathcal{B}}(m)=\frac{1}{2\pi i}\int_{\alpha-i T}^{\alpha+i T} \frac{\left(\frac{Q}{d}+\frac{1}{2} \right)^s F(s)}{s}ds+R(T), \numberthis\label{53}\]
and
\[|R(T)|\ll \frac{Q^\alpha}{d T}\sum_{n=1}^\infty \frac{n^{-\alpha}}{\left|\text{log}\frac{\frac{Q}{d}+\frac{1}{2}}{n}\right|}.\] 
Let $T\ge 2$, $\alpha =\frac{3}{2}.$ Since, $|\chi(m)\mu_{\mathcal{B}}(m)|\le 1$, the bound for $R(T)$ obtained on page 106-107 of \cite{MR1790423} or \cite{MR2342669} also works in our case. Therefore, 
\[|R(T)|\ll \frac{Q^{\frac{3}{2}}\text{log}Q}{d^{\frac{3}{2}}T}. \numberthis\label{54}\]
We now consider the integral in \eqref{53}, which we estimate by shifting the line of integration into a rectangular contour with line segments joining the point  $3/2-i T$, $3/2+i T$, $\theta + i T,$ and $\theta -i T$.  %to include the simple pole of the integrand in \eqref{53} at $s=1$. 
We first consider the principle character $\chi_0$ modulo $k$. In this case,
\begin{align*}
F(s)&=\sum_{n=1}^\infty \frac{\chi_0(n)\mu_{\mathcal{B}}(n)}{n^s} =\prod_p \left[1+\frac{\chi_0(p)\mu_{\mathcal{B}}(p)}{p^s}+\frac{\chi_0(p^2)\mu_{\mathcal{B}}(p^2)}{p^{2s}}+\cdots \right]\\
&=\prod_{\substack{p\not\in \mathcal{B}\\ p\nmid k}} \left[1+\frac{1}{p^s}+\frac{1}{p^{2s}}+\cdots \right]
=\prod_{\substack{p\not\in \mathcal{B}\\ p\nmid k}}\left(\frac{1}{1-\frac{1}{p^s}} \right)=\zeta(s)\prod_{\substack{p\in \mathcal{B}\\ p\nmid k}}\left(1-\frac{1}{p^s}\right)\prod_{p|k}\left(1-\frac{1}{p^s}\right),\numberthis{\label{a1}}
\end{align*}
having analytic continuation to the half plane $\Re(s)>\theta.$ Moreover, the term containing the product on the far right side of \eqref{a1} is bounded in any half plane $\Re(s)>\sigma$ with $\sigma>\theta.$ Also, there is a constant $P_{k,\mathcal{B}}>0$ such that
\[\left|\prod_{\substack{p\in \mathcal{B}\\ p\nmid k}}\left(1-\frac{1}{p^s}\right)\right|\leq P_{k,\mathcal{B}}. \]
By Cauchy's residue theorem, we have
\[\frac{1}{2\pi i}\int_{3/2-i T}^{3/2+i T} \frac{\left(\frac{Q}{d}+\frac{1}{2} \right)^s F(s)}{s}ds = \frac{Q}{d }\prod_{\substack{p\in \mathcal{B}\\ p\nmid k}}\left(1-\frac{1}{p}\right)\prod_{p|k}\left(1-\frac{1}{p}\right) +\sum_{j=1}^{3}I_j, \numberthis\label{integral1}\]
where $I_1$ and $I_3$ are the integrals along the horizontal segments $[3/2-i T,\theta-i T]$ and $[\theta+i T,3/2+i T],$ respectively and $I_2$ is the integral over the vertical segment $[\theta-i T,\theta+i T].$ The first term in the right hand side is obtained from the residue of the simple pole at $s = 1$, giving rise to the main term in the asymptotic formula in the statement of Theorem \ref{theorem3}.
To estimate the $I_j$'s, one finds that the bounds provided in \eqref{zetabound}, for similar integrals, apply to our case as well, modulo multiplication by constants depending on $k$ and $\mathcal{B}$.
Therefore,
\begin{align*}
|I_1|, |I_3|
&\ll_{k,\mathcal{B}}\frac{\text{log}Q\log ^2T}{d^{3/2}T^{1/2}}\left(\frac{Q}{dT^{1/2}}-\frac{Q^{\theta}}{d^{\theta}T^{\theta/2}}+\frac{Q^{3/2}}{d^{3/2}}-\frac{Q}{d}\right),\numberthis\label{55}
\end{align*}
and
\begin{align*}
   |I_2|&\ll_{k,\mathcal{B}}\int_{-T}^{T}\frac{|\left(\frac{Q}{d}+\frac{1}{2}\right)^{\theta+i t}||\zeta(\theta+i t)|}{|\theta+ i t|}dt \ll_{k,\mathcal{B}}\left(\frac{Q}{d}\right)^{\theta} \int_{0}^{T}\frac{|\zeta(\theta+i t)|}{|\theta+ i t|}dt \\&\ll_{k,\mathcal{B}} (Q/d)^{\theta}\text{log}T,\numberthis\label{i2}
\end{align*}
where we used Lemma \ref{zetaintegral} to estimate the zeta integral.
Next, we consider the case for a non-principle character $\chi\neq\chi_0$. We continue with the rectangular contour defined above. Using bounds for $L(s, \chi)$, we obtain
    \[|I_1|,|I_3|  \ll_{k,\mathcal{B}}\frac{Q^{(3/2)}\log Q \log T}{d^{(3/2)}T} \ ; \ |I_2|\ll_{k,\mathcal{B}} \frac{Q^{\theta}(\log T)^2}{d^{\theta}T^{\theta}}.\]
Therefore, inserting \eqref{55} and \eqref{i2} into \eqref{integral1} and choosing $T=Q^2,$ for $\chi=\chi_0$, we have
\[S_{11}= \frac{Q}{d}\prod_{\substack{p\in \mathcal{B}\\p\nmid k}}\left(1-\frac{1}{p} \right)\prod_{p|k}\left(1-\frac{1}{p}\right) + \BigOkb{\frac{Q^{\theta}\log Q}{d^{\theta}}},\numberthis \label{principal}\]
and for $\chi\neq\chi_0$, we obtain
\[S_{11}=\BigOkb
{\frac{(\log Q)^{2}}{d^{3/2}}}  .\numberthis\label{4.18}\]
For $\chi=\chi_0$, invoking \eqref{principal} in \eqref{52} gives \begin{align*}
 S_1&=\sum_{\substack{s\leq Q\\ \mu_{\mathcal{B}}(s)=1}}\chi_0(s)\sum_{d|s}\mu(d)\left\lfloor{\frac{Q}{d}}\right\rfloor={Q^2}\prod_{\substack{p\in \mathcal{B}\\p\nmid k}}\left(1-\frac{1}{p} \right)\prod_{p|k}\left(1-\frac{1}{p}\right)\sum_{d\leq Q}\frac{\chi_0(d)\mu(d)\mu_{\mathcal{B}}(d)}{d^2}\\
 &+\BigOkb{Q^{1+\theta}\text{log}Q\sum_{d\leq Q}\frac{1}{d^{1+\theta}} }. \numberthis\label{57}
 \end{align*}
 One can write
\[\sum_{d\leq Q}\frac{\chi_0(d)\mu(d)\mu_{\mathcal{B}}(d)}{d^2}=\sum_{d=1 }^{\infty}\frac{\chi_0(d)\mu(d)\mu_{\mathcal{B}}(d)}{d^2}+\BigOkb{\frac{1}{Q} },\numberthis\label{58} \] 
and
\begin{align*}
   \sum_{d=1 }^{\infty}\frac{\chi_0(d)\mu(d)\mu_{\mathcal{B}}(d)}{d^2}&=\prod_p \left[1+\frac{\chi_0(p)\mu(p)\mu_{\mathcal{B}}(p)}{p^2}+\frac{\chi_0(p^2)\mu(p^2)\mu_{\mathcal{B}}(p^2)}{p^4}+\cdots \right] =\prod_p \left(1-\frac{\chi_0(p)\mu_{\mathcal{B}}(p)}{p^2} \right)
   \\&=\prod_{\substack{p\not\in\mathcal{B}\\p\nmid k}}\left(1-\frac{1}{p^2} \right)=\frac{1}{\zeta(2)}\prod_{\substack{p\in\mathcal{B}\\p\nmid k}}\left(1-\frac{1}{p^2} \right)^{-1}\prod_{p|k}\left(1-\frac{1}{p^2} \right)^{-1}. \numberthis\label{f11}
\end{align*}
Now, using \eqref{57} and \eqref{f11}, for the principle character $\chi=\chi_0$, we have
\[S_1=\frac{Q^2}{ \zeta(2)}\prod_{\substack{p\in \mathcal{B}\\p\nmid k}}\left(1+\frac{1}{p} \right)^{-1}\prod_{p|k}\left(1+\frac{1}{p} \right)^{-1}+\BigOkb{Q^{1+\theta}\log Q}.\numberthis\label{result1}\]
Likewise for $\chi\neq\chi_0$, using \eqref{4.18} in \eqref{52} gives
 \[S_1\ll \BigOkb
{Q(\log Q)^2}.\]
Next, we first examine the second sum $S_2$ in \eqref{51} for $\chi=\chi_0$. Employing the formula $\phi(n)=\sum_{d|n}\mu(d)n/d,$ we have
\begin{align*}
    S_2&=\sum_{\substack{s\leq Q\\ \mu_{\mathcal{B}}(s)=1}}\chi_0(s)\sum_{d|s}\frac{\mu(d)s}{d} 
    =\sum_{d\leq Q}\frac{\mu(d)}{d}\sum_{\substack{s\leq Q\\ d|s}}\chi_0(s)\mu_{\mathcal{B}}(s)s \\
    &=\sum_{d\leq Q}\chi_0(d)\mu(d)\mu_{\mathcal{B}}(d)\sum_{m\leq \frac{Q}{d}}\chi_0(m)\mu_{\mathcal{B}}(m)m. \numberthis\label{59} 
\end{align*}
Using partial summations on the inner sum on the far RHS $\eqref{59}$ and then applying $\eqref{principal}$, we have
\[\sum_{m\leq \frac{Q}{d}}\chi_0(m)\mu_{\mathcal{B}}(m)m= \frac{Q^2}{2d^2}\prod_{\substack{p\in \mathcal{B}\\p\nmid k}}\left(1-\frac{1}{p} \right)\prod_{p|k}\left(1-\frac{1}{p} \right)+\BigOkb{\frac{Q^{1+\theta}(\text{log}Q)^{3/2}}{d^{1+\theta}}}.\numberthis\label{a6}\]
This along with \eqref{58}, \eqref{f11} and \eqref{59} yields an estimate for $S_2$ for the principal character $\chi_0$
\begin{align*}
   S_2=&{\frac{Q^2}{2}}\prod_{\substack{p\in \mathcal{B}\\p\nmid k}}\left(1-\frac{1}{p} \right)\prod_{p|k}\left(1-\frac{1}{p} \right)\sum_{d\leq Q}\frac{\chi_0(d)\mu(d)\mu_{\mathcal{B}}(d)}{d^2}\\ &+\BigOkb{Q^{1+\theta}(\log Q)^{3/2}\sum_{d\leq Q}\frac{1}{d^{1+\theta}} }\\
   &= \frac{Q^2}{2\zeta(2)}\prod_{\substack{p\in \mathcal{B}\\p\nmid k}}\left(1+\frac{1}{p} \right)^{-1}\prod_{p|k}\left(1+\frac{1}{p} \right)^{-1}+\BigOkb{Q^{1+\theta}({\log}Q)^{3/2}}.\numberthis\label{f12}
   \end{align*}
   For all other characters, an estimate is as below
\begin{align*}
   S_2&
    =\sum_{d\leq Q}\chi(d)\mu(d)\mu_{\mathcal{B}}(d)\sum_{m\leq \frac{Q}{d}}\chi(m)\mu_{\mathcal{B}}(m)m=\sum_{d\leq Q}\chi(d)\mu(d)\mu_{\mathcal{B}}(d)\sum_{j\leq \frac{Q}{d}}\sum_{j\leq m\leq \frac{Q}{d}}\chi(m)\mu_{\mathcal{B}}(m)  \\
    &\ll_{k,\mathcal{B}}{Q}(\log Q)^2 \sum_{d\leq Q}\frac{1}{d^{5/2}} \ll_{k,\mathcal{B}}{Q}(\log Q)^2. \numberthis\label{4.19}
\end{align*}
 Inserting the estimates from \eqref{result1}, \eqref{f12}, \eqref{4.18}, and \eqref{4.19} into \eqref{51} proves Theorem \ref{theorem3}.
\section{Second Moment}
\subsection{Deficiency}
To estimate the second moment, we need an asymptotic formula for the deficiency sum in $\eqref{76}$. In this section, we discuss average order of deficiency when the denominator of the Farey fractions is  $\mathcal{B}$-free. %and belongs to some arithmetic progression.
\begin{thm}\label{Thmdeficiency}
For a fixed  positive integer $k$, let $\chi$ be a Dirichlet character modulo k and a set $\mathcal{B}$ of prime numbers such that $\sum_{p\in\mathcal{B}}\dfrac{1}{p^\sigma}<\infty$ for some $\sigma<\theta,$ where $1/2<\theta<1$. Then, for all large positive integers $Q,$ we have
\begin{displaymath}
\sum_{\substack{s\leq Q\\ \mu_{\mathcal{B}}(s)=1}}\chi(s)\delta(s)=\left
\{\begin{array}{lr}
   \BigOkb{Q(\log Q)^4},  & \chi\ne \chi_0  \\ \\
 {2Q^2}{\prod_{\substack{p\in \mathcal{B}\\p\nmid k}}\left(1+\frac{1}{p}\right)^{-1}}\prod_{p|k}\left(1+\frac{1}{p}\right)^{-1}-\frac{3Q^2}{\zeta(2)}\prod_{\substack{p\in \mathcal{B}\\p\nmid k}}\left(1+\frac{1}{p} \right)^{-1} \\ \times\prod_{p|k}\left(1+\frac{1}{p}\right)^{-1} +  \BigOkb{Q^{1+\theta}(\log Q)^2},  & \chi =\chi_0.
\end{array} 
\right .
\end{displaymath}
\end{thm}
\begin{proof}
Using $\eqref{42}$, we have
\[\sum_{\substack{s\leq Q\\ \mu_{\mathcal{B}}(s)=1}}\chi(s)\delta(s)=\sum_{\substack{s\leq Q\\ \mu_{\mathcal{B}}(s)=1}}\chi(s)\phi(s)\left\lfloor{\frac{2Q}{s}} \right\rfloor-\sum_{\substack{s\leq Q\\ \mu_{\mathcal{B}}(s)=1}}\chi(s)T(s). \numberthis\label{deficiency}\]
We begin with estimation of the first sum on the right side of the above equation. 
\\ Case (i): For $\chi\ne \chi_0$, we write
\begin{align*}
  \sum_{\substack{s\leq Q\\ \mu_{\mathcal{B}}(s)=1}}\chi(s)\phi(s)\left\lfloor{\frac{2Q}{s}}\right\rfloor &= \sum_{\substack{s\leq Q\\ \mu_{\mathcal{B}}(s)=1}}\chi(s)\left\lfloor{\frac{2Q}{s}}\right\rfloor\sum_{d|s}\mu(d)\frac{s}{d}
  =\sum_{d\leq Q}\frac{\mu(d)}{d}\sum_{\substack{s\leq Q\\ d|s}}\chi(s)\mu_{\mathcal{B}}(s)s\left\lfloor{\frac{2Q}{s}}\right\rfloor\\
   &=\sum_{d\leq Q}\chi(d)\mu(d)\mu_{\mathcal{B}}(d)\sum_{m\leq \frac{Q}{d}}\chi(m)\mu_{\mathcal{B}}(m) m\left\lfloor{\frac{2Q}{md}} \right\rfloor\\
   &=\sum_{d\leq Q}\chi(d)\mu(d)\mu_{\mathcal{B}}(d)U_{\chi}\left(\frac{2Q}{d} \right), \numberthis\label{5.2}
 \end{align*}
 where
 \[U_{\chi}(y)=\sum_{l\leq \frac{y}{2}}\chi(l)\mu_{\mathcal{B}}(l)l\left\lfloor{\frac{y}{l}} \right\rfloor. \]
Define $A(n)=\sum_{m\leq n}\chi(m)\mu_{\mathcal{B}}(m)$ and $b(n)=n\left\lfloor{\dfrac{y}{n}} \right\rfloor$ if $n\leq \left\lfloor{\dfrac{y}{2} }\right\rfloor$ and $0$ otherwise. 
 We have
 \begin{align*}
     U_{\chi}(y)&=\sum_{n\leq \left\lfloor{\frac{y}{2}}\right\rfloor}\chi(n)\mu_{\mathcal{B}}(n)b(n)=\sum_{n\leq \left\lfloor{\frac{y}{2}}\right\rfloor}(A(n)-A(n-1))b(n)=\sum_{n\leq \left\lfloor{\frac{y}{2}}\right\rfloor}A(n)b(n)\\&-\sum_{n\leq \left\lfloor{\frac{y}{2}}\right\rfloor-1}A(n)b(n+1)=\sum_{n\leq \left\lfloor{\frac{y}{2}}\right\rfloor}A(n)(b(n)-b(n+1)).
 \end{align*}
 The sum $A(n)$ is similar to $S_{11}$ where the former sum is upto $n$ and the latter upto $Q/d$. Hence, from \eqref{4.18},  $A(n)\ll_{k,\mathcal{B}}(\log n)^{2} $ and  
 \begin{align*}
    |b(n)-b(n+1)|&\leq \frac{y}{n+1}+n\left(\left\lfloor{\frac{y}{n}} \right\rfloor-\left\lfloor{\frac{y}{n+1}} \right\rfloor \right).
 \end{align*} %\numberthis\label{5.3}
This yields
 \begin{align*}
     U_{\chi}\left(\frac{2Q}{d}\right)&\ll_{k,\mathcal{B}}\frac{Q}{d}\sum_{n\leq \left\lfloor{\frac{Q}{d}} \right\rfloor}\frac{(\log n)^{2}}{n}
     \ll_{k,\mathcal{B}}\frac{Q}{d}(\log Q)^{3}.
 \end{align*}
 \iffalse
 Hence
 \[ U_{\chi}\left(\frac{2Q}{d}\right)\ll_{k,\mathcal{B}}\frac{Q(\log Q)^{3}}{d}.\]
 \fi
 Combining this with $\eqref{5.2}$, we have
     \[\sum_{\substack{s\leq Q\\ \mu_{\mathcal{B}}(s)=1}}\chi(s)\phi(s)\left\lfloor{\frac{2Q}{s}}\right\rfloor\ll_{k,\mathcal{B}}Q(\log Q)^{3}\sum_{d\leq Q}\frac{1}{d}\ll_{k,\mathcal{B}}Q(\log Q)^{4}.\numberthis\label{81}\]
     Inserting the above bound into $\eqref{deficiency}$ and applying Theorem \ref{theorem3} gives the required result for the case $\chi\ne \chi_0.$ \\
  Case (ii):    
For $\chi=\chi_0$, %since $\lfloor{2Q/s}\rfloor=1$ if $Q<s<2Q,$ so 
We extend the range from $1\leq s\leq Q$ to $1\leq s\leq 2Q,$ and write
\[\sum_{\substack{s\leq Q\\ \mu_{\mathcal{B}}(s)=1}}\chi_0(s)\phi(s)\left[\frac{2Q}{s}\right\rfloor=\sum_{\substack{s\leq 2Q\\ \mu_{\mathcal{B}}(s)=1}}\chi_0(s)\phi(s)\left\lfloor{\frac{2Q}{s}}\right\rfloor-\sum_{\substack{Q< s\leq 2Q\\ \mu_{\mathcal{B}}(s)=1}}\chi_0(s)\phi(s). \numberthis\label{72}\]
From \eqref{f12}, we have
\begin{align*}
    \sum_{\substack{Q< s\leq 2Q\\ \mu_{\mathcal{B}}(s)=1}}\chi_0(s)\phi(s)&= \frac{3Q^2}{2\zeta(2)}\prod_{\substack{p\in \mathcal{B}\\p\nmid k}}\left(1+\frac{1}{p} \right)^{-1}\prod_{p|k}\left(1+\frac{1}{p} \right)^{-1}+\BigOkb{Q^{1+\theta}(\log Q)^{\frac{3}{2}}}.\numberthis\label{principalphi}
\end{align*}
\iffalse
Substituting the above result in $\eqref{72}$, we get
\begin{align*}
\sum_{\substack{s\leq Q\\ \mu_{\mathcal{B}}(s)=1}}\chi_0(s)\phi(s)\left\lfloor{\frac{2Q}{s}}\right\rfloor=&\sum_{\substack{s\leq 2Q\\ \mu_{\mathcal{B}}(s)=1}}\chi_0(s)\phi(s)\left\lfloor{\frac{2Q}{s}}\right\rfloor-\frac{3Q^2}{2\zeta(2)}\prod_{\substack{p\in \mathcal{B}\\p\not|\;k}}\left(1+\frac{1}{p} \right)^{-1}\prod_{p|k}\left(1+\frac{1}{p} \right)^{-1} \\
&+\BigOkb{Q^{1+\theta}(\log Q)^{\frac{3}{2}}}. \numberthis\label{75}
\end{align*}
\fi
The first sum on the right side is solved as
\begin{align*}
    \sum_{\substack{s\leq 2Q\\ \mu_{\mathcal{B}}(s)=1}}\chi_0(s)\phi(s)\left\lfloor{\frac{2Q}{s}}\right\rfloor&=\sum_{\substack{s\leq 2Q\\ \mu_{\mathcal{B}}(s)=1}}\chi_0(s)\phi(s)\sum_{\substack{n\leq 2Q \\ s|n}}1 = \sum_{n\leq 2Q}\sum_{s|n}\chi_0(s)\phi(s)\mu_{\mathcal{B}}(s) = \sum_{n\leq 2Q}f_{\chi_0}(n), 
\end{align*}
%\numberthis\label{deficiency1}
where
\[f_{\chi_0}(n):=\sum_{m|n}\chi_0(m)\phi(m)\mu_{\mathcal{B}}(m). \]
\iffalse
We write
\begin{align*}
    G_{\chi}(s)&=\sum_{n=1}^{\infty}\frac{\chi_0(n)\phi(n)\mu_{\mathcal{B}}(n)}{n^s} =\prod_{p}\left(1+\frac{\chi_0(p)\phi(p)\mu_{\mathcal{B}}(p)}{p^s}+\frac{\chi_0(p^2)\phi(p^2)\mu_{\mathcal{B}}(p^2)}{p^{2s}}+\cdots \right) \\
    &=\prod_{\substack{p\not\in \mathcal{B}\\p\not|\;k}}\left(1+\frac{\phi(p)}{p^s}+\frac{\phi(p^2)}{p^{2s}}+\cdots \right) 
    %&=\prod_{p}\left(1+\frac{\phi(p)}{p^s}+\frac{\phi(p^2)}{p^{2s}}+\cdots \right) \prod_{\substack{p\in \mathcal{B}\\p|k}}\left(1+\frac{\phi(p)}{p^s}+\frac{\phi(p^2)}{p^{2s}}+\cdots \right)^{-1} \\
    = \frac{\zeta(s-1)}{\zeta(s)}\prod_{\substack{p\in \mathcal{B}\\p\not|\;k}}\left({1-\frac{1}{p^{s-1}}}\right)\left({1-\frac{1}{p^s}}\right)^{-1}\prod_{p|k}\left({1-\frac{1}{p^{s-1}}}\right)\left({1-\frac{1}{p^s}}\right)^{-1}.
\end{align*}
\fi
Since $f_{\chi_0}=\chi_0\phi\mu_{\mathcal{B}}*1$, the Dirichlet series for $f_{\chi_0}$ is given by
\begin{align*}
    F(s)&=\sum_{n=1}^{\infty}\frac{f_{\chi_0}(n)}{n^s} =\left(\sum_{n=1}^{\infty}\frac{1}{n^s}\right)\left(\sum_{n=1}^{\infty}\frac{\chi_0(m)\phi(m)\mu_{\mathcal{B}}(m)}{n^s} \right) \\
     &=\zeta(s-1)\prod_{\substack{p\in \mathcal{B}\\p\nmid k}}\left({1-\frac{1}{p^{s-1}}}\right)\left({1-\frac{1}{p^s}}\right)^{-1}\prod_{p|k}\left({1-\frac{1}{p^{s-1}}}\right)\left({1-\frac{1}{p^s}}\right)^{-1}.
\end{align*}
which is absolutely convergent for $\Re(s)>2$. Moreover, the product term over primes in $\mathcal{B}$ is bounded for $\Re(s)>1+\theta$.
Now, use Lemma \ref{parronlemma} with $x=2Q+\frac{1}{3}$, and sum over $n\leq 2Q$, we get
\[\sum_{n\leq 2Q} f_{\chi_0}(n)=\frac{1}{2\pi i}\int_{\alpha-i T}^{\alpha+i T} \frac{\left(2Q+\frac{1}{3} \right)^s F(s)}{s}ds+R(T), \numberthis\label{74}\]
where
\[|R(T)|\ll \frac{Q^\alpha}{T}\sum_{n=1}^\infty \frac{n}{n^{\alpha}\left|\text{log}\frac{2Q+\frac{1}{3}}{n}\right|} .\]
Let $\alpha =2+\frac{1}{\log Q}$ and using the same argument as in \eqref{54}, we have
%Using the arguments given in \cite{MR1790423}  (see pp. 106-107),
%We decompose the above sum into three subsums extended over the following sets of $n: n\leq Q, Q<n\leq 3Q,$ and $n>3Q.$ For values of $n$ which satisfy $n\leq Q $ or $n>3Q$, it is immediately clear that
%\[\left|\text{log}\frac{2Q+\frac{1}{3}}{n} \right|>\text{log}\frac{3}{2}.\]
%Hence the first and last subsums are $\ll 1.$ For values of $n$ which satisfy $Q<n\leq 3Q $, the middle subsum is
%\[\ll Q^{1-\alpha}\sum_{-Q<m\leq Q} \frac{1}{\left|\text{log}\frac{2Q+\frac{1}{3}}{2Q+m} \right|}\ll Q^{2-\alpha}\sum_{-Q<m\leq Q}\frac{1}{\left|m-\frac{1}{2} \right|}\ll Q^{2-\alpha}\text{log}Q. \]
\[|R(T)|\ll \frac{Q^2\text{log}Q}{T}.\]
The integrand in $\eqref{74}$ has a simple pole at $s=2.$ We deform the line integral into a rectangular contour with vertices $\alpha\pm iT$ and $1\pm iT.$ By Cauchy's Residue theorem, we have
\[\frac{1}{2\pi i}\int_{\alpha-i T}^{\alpha+i T} \frac{\left(2Q+\frac{1}{3} \right)^s F(s)}{s}ds =2^{-1} \left(2Q+\frac{1}{3} \right)^2\prod_{\substack{p\in \mathcal{B}\\p\nmid k}}\left(1+\frac{1}{p}\right)^{-1}\prod_{p|k}\left(1+\frac{1}{p}\right)^{-1} +\sum_{j=1}^{3}I_j, \]
where the first term of right side is due to the residue of simple pole at the point $s=2$ and $I_1$, $I_2$, and  $I_3$ are the integrals along the lines $[\alpha-i T,1-i T]$, $[1-iT,1+iT],$ and $[1+i T,\alpha+i T]$, respectively.
\\ Estimation of $I_1$ and $I_3$: 
Using  \eqref{zetabound}, we have 
\begin{align*}
    |I_1|,|I_3|
    &\ll_{k,\mathcal{B}}\int_{1}^{\alpha}\frac{|\left(2Q+\frac{1}{3}\right)^{\sigma+i T}||\zeta(\sigma-1+iT)|}{|\sigma+ i T|}d\sigma 
    \ll_{k,\mathcal{B}}\log T\int_{1}^{\alpha} (QT^{-1/2})^{\sigma}d\sigma \\
    &\ll_{k,\mathcal{B}}\frac{(Q^2+QT^{1/2})\log T \log (QT^{-1/2})}{T}.
\end{align*}
Estimation of $I_2$: Since $\zeta(s-1)\ll T^{(1-\sigma)/2} \log T$ if $1\leq \sigma \leq 2,$ we have
\begin{align*}
    |I_2|&\ll_{k,\mathcal{B}}\int_{-T}^{T}\frac{|(2Q+\frac{1}{3})^{1+it}||\zeta(it)|}{|1+it|}dt 
    \ll_{k,\mathcal{B}}QT^{1/2} \log T \int_{0}^{T}\frac{1}{|1+it|}dt 
    \ll_{k,\mathcal{B}} QT^{1/2} (\log T)^2.
\end{align*}
Collecting all estimates and choosing $T=Q^{2/3},$ we get
\[\sum_{n\leq 2Q}f_{\chi_0}(n)=2Q^2\prod_{\substack{p\in \mathcal{B}\\p\nmid k}}\left(1+\frac{1}{p}\right)^{-1}\prod_{p|k}\left(1+\frac{1}{p}\right)^{-1}+\BigOkb{Q^{4/3}(\log Q)^2}.\numberthis\label{80}\]
We get the required result for $\chi_0$ by inserting \eqref{80} and \eqref{principalphi} in $\eqref{72}.$
\end{proof}
\subsection{Proof of second moment}
\begin{proof}
To find the second moment, we first expand the terms $X_{\chi}(Q)$ and $Y_{\chi}(Q)$ in \eqref{76} for any character $\chi$ in a manner similar to \cite{MR2424917}, but now for the case of $\mathcal{B}-$ free numbers. 
 %Since $\lfloor{2Q/s}\rfloor=1$ when $Q<s\leq 2Q,$ we extend the range from $1\leq s \leq Q$ to $1\leq s \leq 2Q$ and we have
%\[X_{\chi}(Q)=\sum_{\substack{s\leq 2Q \\ \mu_{\mathcal{B}}(s)=1}}\chi(s)\phi(s){\left \lfloor {\frac{2Q}{s}}\right \rfloor}^2- \sum_{\substack{Q<s\leq 2Q \\ \mu_{\mathcal{B}}(s)=1}}\chi(s)\phi(s).\numberthis\label{82}\]
\begin{align*}
X_{\chi}(Q)&=2\sum_{n\leq 2Q}n h_{\chi}(n)-\sum_{\substack{s\leq 2Q \\ \mu_{\mathcal{B}}(s)=1}}\chi(s)\phi(s){\left \lfloor {\frac{2Q}{s}}\right \rfloor} - \sum_{\substack{Q<s\leq 2Q \\ \mu_{\mathcal{B}}(s)=1}}\chi(s)\phi(s),\numberthis\label{831}
\end{align*}
where
\[h_{\chi}(n)=\sum_{s|n}\frac{\chi(s)\phi(s)\mu_{\mathcal{B}}(s)}{s}.\numberthis\]
And similarly,
\begin{align*}
    Y_{\chi}(Q)&=2\sum_{n\leq 2Q}(n-Q) h_{\chi}(n)-2\sum_{\substack{Q<s\leq 2Q \\ \mu_{\mathcal{B}}(s)=1}}\chi(s)\phi(s)+2Q\sum_{\substack{Q<s\leq 2Q \\ \mu_{\mathcal{B}}(s)=1}}\frac{\chi(s)\phi(s)}{s}+ \BigO{Q(\log Q)^2}.\numberthis\label{90}
\end{align*}
From \eqref{76}, \eqref{831}, and \eqref{90}, we have
\begin{align*}
\mathcal{M}_{2,\mathcal{B}}(\chi,Q)=&2\sum_{n\leq 2Q}(2Q-n)h_{\chi}(n)-\sum_{\substack{s\leq 2Q \\ \mu_{\mathcal{B}}(s)=1}}\chi(s)\phi(s){\left \lfloor {\frac{2Q}{s}}\right \rfloor}+3\sum_{\substack{Q<s\leq 2Q \\ \mu_{\mathcal{B}}(s)=1}}\chi(s)\phi(s)\\ &-4Q\sum_{\substack{Q<s\leq Q \\ \mu_{\mathcal{B}}(s)=1}}\frac{\chi(s)\phi(s)}{s}+\sum_{\substack{s\leq Q \\ \mu_{\mathcal{B}}(s)=1}}\chi(s)\delta(s) +\BigO{Q(\log Q)^2}\\&=2M_1-M_2+3M_3-4QM_4+M_5+\BigO{Q(\log Q)^2}.\numberthis\label{93} 
\end{align*}
One can obtain $M_5$ from Theorem \ref{Thmdeficiency}. Estimation of $M_2$ is in \eqref{81} and \eqref{80} for non principal and principal characters, respectively. \\ Case (i): For $\chi\ne\chi_0$, from \eqref{4.19}, we have \[M_3\ll Q\log^2Q, \numberthis\label{M3nonprin}\] and
\begin{align*}
   M_4&=\sum_{\substack{Q<s\leq Q \\ \mu_{\mathcal{B}}(s)=1}}\frac{\chi(s)}{s}\sum_{d|s}\frac{\mu(d)s}{d} 
    =\sum_{d\leq 2Q}\frac{\chi(d)\mu(d)\mu_{\mathcal{B}}(d)}{d}\sum_{\frac{Q}{d}<m\leq \frac{2Q}{d}}\chi(m)\mu_{\mathcal{B}}(m) \\
    &\ll_{k,\mathcal{B}}(\log Q)^{2}\sum_{d\leq 2Q}\frac{1}{d^{\frac{5}{2}}}
    \ll_{k,\mathcal{B}}(\log Q)^{2}.\numberthis\label{s100}
\end{align*}
Case (ii): For $\chi=\chi_0$, from \eqref{f12} 
\[M_3=\frac{3Q^2}{2\zeta(2)}\prod_{\substack{p\in \mathcal{B}\\p\nmid k}}\left(1+\frac{1}{p} \right)^{-1}\prod_{p|k}\left(1+\frac{1}{p} \right)^{-1}+\BigOkb{Q^{1+\theta}({\log}Q)^{3/2}},\numberthis\label{M3prin}\]
and using \eqref{principal}, \eqref{58}, and \eqref{f11}, we obtain 
\begin{align*}
   M_4&=\frac{Q}{\zeta(2)}\prod_{\substack{p\in \mathcal{B}\\p\nmid k}}\left(1+\frac{1}{p} \right)^{-1}\prod_{p|k}\left(1+\frac{1}{p} \right)^{-1}+\BigOkb{Q^{\theta}(\log Q)^{\frac{3}{2}}}.\numberthis\label{M4}
\end{align*}
We are left to estimate $M_1$. For this we observe that 
\[M_1=\sum_{n=1}^{\infty}\text{max}\left(\frac{2Q}{n}-1,0 \right)nh_{\chi}(n),\] where $h_{\chi}(n)$ is in \eqref{93}. Using the well known formula \[\frac{1}{2\pi i}\int_{\alpha-i\infty}^{\alpha+i\infty}\frac{x^{s+1}}{s(s+1)}ds=\text{max}(x-1,0),\  x>0 , \numberthis\label{int}\] we write 
\[M_1=\frac{1}{2\pi i}\sum_{n=1}^{\infty}\int_{\alpha-i\infty}^{\alpha+i\infty}\frac{(2Q)^{s+1}}{s(s+1)}\left(\frac{h_{\chi}(n)}{n^s} \right)ds.\]
 The Dirichlet series of $h_{\chi}(n)$ is absolutely convergent on the line $\Re(s)=2$ and is given by
\begin{align*}
    H(s)&=\sum_{n=1}^{\infty}\frac{h_{\chi}(n)}{n^s} =\zeta(s)\sum_{n=1}^{\infty}\frac{\chi(n)\phi(n)\mu_{\mathcal{B}}(n)}{n^{s+1}}=\zeta(s)\prod_{p\not \in \mathcal{B}}\left[1+\frac{\chi(p)\phi(p)}{p^{s+1}} +\frac{\chi(p^2)\phi(p^2)}{(p^2)^{s+1}}+\cdots\right] \\
    &= \frac{\zeta(s)L(s,\chi)}{L(s+1,\chi)}\prod_{p\in \mathcal{B}}\left[\sum_{n=0}^{\infty}\frac{\chi(p^n)\phi(p^n)}{p^{ns+n}} \right]^{-1}=\frac{\zeta(s)L(s,\chi)}{L(s+1,\chi)}\prod_{p\in \mathcal{B}}\left({1-\frac{\chi(p)}{p^{s}}}\right)\left({1-\frac{\chi(p)}{p^{s+1}}}\right)^{-1}.
\end{align*}
This yields
\begin{align*}
    M_1=\frac{1}{2\pi i}\int_{\alpha-i\infty}^{\alpha+i\infty}\frac{(2Q)^{s+1}\zeta(s)L(s,\chi)}{s(s+1)L(s+1,\chi)}\prod_{p\in \mathcal{B}}\left({1-\frac{\chi(p)}{p^{s}}}\right)\left({1-\frac{\chi(p)}{p^{s+1}}}\right)^{-1}ds . \numberthis\label{s101}
\end{align*}
Case (i): For $\chi\ne \chi_0$, the above integrand has simple poles at points $s=0$ and $s=1,$ so  we shift the path of integration from $2-iT$ to $2+iT$ into a contour that contains the  horizontal line segments from $2-iT$ to $\theta-iT$, and from  $\theta+iT$ to $2+iT$ and the vertical line segments from $\theta-iT$ to $\theta+iT,$ and from  $2-iT$ to $2+iT$.  By Cauchy's residue theorem, we have
\begin{align*}
   \frac{1}{2\pi i}\int_{2-i\infty}^{2+i\infty}\frac{(2Q)^{s+1}\zeta(s)L(s,\chi)}{s(s+1)L(s+1,\chi)}\prod_{p\in \mathcal{B}}\left({1-\frac{\chi(p)}{p^{s}}}\right)\left({1-\frac{\chi(p)}{p^{s+1}}}\right)^{-1}ds \\=\frac{2Q^2L(1,\chi)}{L(2,\chi)}\prod_{p\in \mathcal{B}}\left({1-\frac{\chi(p)}{p}}\right)\left({1-\frac{\chi(p)}{p^{2}}}\right)^{-1} + \sum_{j=1}^{5}I_j,  \numberthis\label{94}
\end{align*}
where the first term is the residue of the integrand at s=1,  $I_1$ and $I_5$ are the integrals along the vertical segments $(2-i\infty,2-iT]$ and   $[2+iT,2+i\infty),$ respectively, $I_2$ and $I_4$ are the integrals along the horizontal axis $[2-it,\theta-iT]$ and $[\theta+iT,2+iT]$, respectively, and  $I_3$ is the integral over the vertical segment $[\theta-iT,\theta+iT].$
\\ Estimation of $I_1$ and $I_5$:
 Since $|\zeta(s)|\ll 1$ and $|L(s,\chi)\ll_k1$ uniformly on $\Re(s)=2$, we obtain
\begin{align*}
   |I_1|,|I_5|&\ll_{k,\mathcal{B}}\int_{T}^{\infty}\frac{Q^3}{t^2}dt \ll_{k,\mathcal{B}}\frac{Q^3}{T}.
\end{align*}
Estimation of $I_2$ and $I_4$:
  Since $|L(\sigma+iT,\chi)|\gg_k \frac{1}{\log T}$ \ (see  \cite{MR2378655}) and using \eqref{zetabound} and \eqref{lbound}, we have
\begin{align*}
    |I_2|,|I_4|&\ll_{k,\mathcal{B}}\log T\int_{\theta}^{2}\frac{Q^{\sigma+1}|\zeta(\sigma+iT)||L(\sigma+iT,\chi)|}{|\sigma+iT||\sigma+1+iT|}d\sigma \\
    &\ll_{k,\mathcal{B}}\frac{(\log T)^3}{T^{\frac{3(1+\theta)}{2}}}\int_{\theta}^{1}Q^{\sigma+1}d\sigma +\frac{(\log T)^3}{T^2}\int_{1}^{2}Q^{\sigma+1}d\sigma \\
     &\ll_{k,\mathcal{B}}\frac{Q^2 \log Q(\log T)^3}{T^{\frac{3(1+\theta)}{2}}} +\frac{Q^3\log Q(\log T)^3}{T^2}.
\end{align*}
Estimation of $I_3$:
Since $|L(\theta+iT,\chi)|\gg_k\frac{1}{\log T}$ (see \cite{MR2378655})  and using \eqref{zetabound}, \eqref{lbound},  we have
\begin{align*}
    |I_3|&\ll_{k,\mathcal{B}}Q^{1+\theta}\log T\int_{0}^{T}\frac{|\zeta(\theta+it)||L(\theta+it,\chi)|}{|\theta+it||1+\theta+it|}dt \\
    &\ll_{k,\mathcal{B}}Q^{1+\theta}T^{\frac{1-3\theta}{2}}(\log T)^3\int_{0}^{T}\frac{1}{|\theta+it||1+\theta+it|}dt \ll_{k,\mathcal{B}}\frac{Q^{1+\theta}(\log T)^3}{T^{\frac{1+3\theta}{2}}}.
   \end{align*}
Collecting all estimates in $\eqref{94}$ and putting $T=Q^2$, we obtain
\[M_1= \frac{2Q^2L(1,\chi)}{L(2,\chi)}\prod_{p\in \mathcal{B}}\left({1-\frac{\chi(p)}{p}}\right)\left({1-\frac{\chi(p)}{p^2}}\right)^{-1}+\BigOkb{Q}.\numberthis\label{M1non}\]
Case (ii): For $\chi=\chi_0$, \eqref{s101} yields
\begin{align*}
  M_1
    &=\frac{1}{2\pi i}\int_{\alpha-i\infty}^{\alpha+i\infty}\frac{(2Q)^{s+1}\zeta(s)^2}{s(s+1)\zeta(s+1)}\prod_{\substack{p\in \mathcal{B}\\p\nmid k}}\left({1-\frac{1}{p^{s}}}\right)\left({1-\frac{1}{p^{s+1}}}\right)^{-1}\prod_{p|k}\left({1-\frac{1}{p^{s}}}\right)\left({1-\frac{1}{p^{s+1}}}\right)^{-1} ds . 
    \end{align*}
Take $\alpha=2.$ We deform the path of integration as defined for $\chi\ne \chi_0$ above. Since the integrand has a pole of order $2$ at $s=1$ inside this contour.
Denote 
\[Z(s):=\frac{(2Q)^{s+1}\zeta^2(s)}{s(s+1)\zeta(s+1)}\prod_{\substack{p\in\mathcal{B}\\p\nmid k}}\left({1-\frac{1}{p^s}}\right)\left({1-\frac{1}{p^{s+1}}}\right)^{-1}\prod_{p|k}\left({1-\frac{1}{p^s}}\right)\left({1-\frac{1}{p^{s+1}}}\right)^{-1}, \]
By Cauchy's residue theorem, we have
\begin{align*}
    &\frac{1}{2\pi i}\int_{2-i\infty}^{2+i\infty}Z(s)ds=\text{Res}_{s=1}Z(s)+\sum_{j=1}^{5}I_j,\numberthis\label{s103}  \end{align*}
where $I_j$'s are same as in case (i) above. 
The residue of the second order pole at $s=1$ of $Z(s)$ is given by
\begin{align*}
\text{Res}_{s= 1}Z(s)=&\frac{2Q^2}{\zeta(2)}\left(\log 2Q-\frac{\zeta^\prime(2)}{\zeta(2)}-\frac{3}{2}+2\gamma  +\sum_{\substack{p\in\mathcal{B}\\p\nmid k}}\frac{p\log p}{p^2-1}+\sum_{p|k}\frac{p\log p}{p^2-1}\right)\prod_{\substack{p\in\mathcal{B}\\p\nmid k}}\frac{p}{p+1}\prod_{p|k}\frac{p}{p+1}.
\end{align*}
Estimation of $I_1$ and $I_5$: Since $|\zeta(s)|\ll 1$ uniformly on $\Re(s)=2$, we obtain
\begin{align*}
   |I_1|, |I_5|&\ll_{k,\mathcal{B}}\int_{T}^{\infty}\frac{Q^3}{t^2}dt 
    \ll_{k,\mathcal{B}}\frac{Q^3}{T}.
\end{align*}
Estimation of $I_2$ and $I_4$: Since   $|\zeta(\sigma+iT)|\gg_k \frac{1}{\log T}$ (see \cite{MR2378655}) and using \eqref{zetabound}, we have
\begin{align*}
   |I_2|,|I_4|&\ll_{k,\mathcal{B}}\log T\int_{\theta}^{2}\frac{Q^{\sigma+1}|\zeta(\sigma+iT)|^2}{|\sigma+iT||\sigma+1+iT|}d\sigma \\
    &\ll_{k,\mathcal{B}}\frac{(\log T)^3}{T^{1+\theta}}\int_{\theta}^{1}Q^{\sigma+1}d\sigma +\frac{(\log T)^3}{T^2}\int_{1}^{2}Q^{\sigma+1}d\sigma \\
     &\ll_{k,\mathcal{B}}\frac{Q^2 \log Q(\log T)^3}{T^{1+\theta}} +\frac{Q^3\log Q(\log T)^3}{T^2}.
\end{align*}
Estimation of $I_3$:  Using \eqref{zetabound}, we have
\begin{align*}
    |I_3|&\ll_{k,\mathcal{B}}Q^{1+\theta}\log T\int_{0}^{T}\frac{|\zeta^2(\theta+it)|}{|\theta+it||1+\theta+it|}dt 
    \ll_{k,\mathcal{B}}Q^{1+\theta}\left(\log T \right)^3\int_{0}^{T}\frac{t^{1-\theta}}{|\theta+it||1+\theta+it|}dt\\
    &\ll_{k,\mathcal{B}}Q^{1+\theta}\left(\log T \right)^3\int_{0}^{T}\frac{1}{(\theta+t)^{1+\theta}}dt
     \ll_{k,\mathcal{B}}\frac{Q^{1+\theta} (\log T)^3}{T^{\theta}}.
   % &\ll_{k,\mathcal{B}}Q^{1+\theta}\log T\left(\int_{0}^{T}\frac{|\zeta(\theta+iT)|^2}{|\theta+iT|^2}dt \right)^{\frac{1}{2}}\left(\int_{0}^{T}\frac{|\zeta(\theta+iT)|^2}{|1+\theta+iT|^2} \right)^{\frac{1}{2}}. \\
\end{align*}
Collecting all estimates in $\eqref{s103}$ and putting $T=Q^2$, we obtain
\begin{align*}
M_1=&\frac{2Q^2}{\zeta(2)}\left(\log 2Q-\frac{\zeta^\prime(2)}{\zeta(2)}-\frac{3}{2}+2\gamma  +\sum_{\substack{p\in\mathcal{B}\\p\nmid k}}\frac{p\log p}{p^2-1}+\sum_{p|k}\frac{p\log p}{p^2-1}\right)\times\prod_{\substack{p\in\mathcal{B}\\p\nmid k}}\left(1+\frac{1}{p} \right)^{-1}\prod_{p|k}\left(1+\frac{1}{p} \right)^{-1}\\&+\BigOkb{Q}.\numberthis\label{M1prin}
\end{align*}
Collecting all estimates from Theorem \ref{Thmdeficiency}, \eqref{81}, \eqref{M3nonprin}, \eqref{s100}, and \eqref{M1non}, we obtain the statement of Theorem \ref{thm4} for non-principal character. For 
$\chi=\chi_0$, the result follows from collecting estimates from Theorem \ref{Thmdeficiency}, \eqref{80}, \eqref{M3prin}, \eqref{M4}, and \eqref{M1prin}.
\end{proof}
\section{Higher Moments}
%In this section, we prove Theorem \ref{thm5}. 
\begin{proof}[Proof of Theorem \ref{thm5}]
From \eqref{highmoment}, we have
\begin{align*}
    \mathcal{M}_{l,\mathcal{B}}(\chi,Q)
    &=\sum_{\substack{s\leq Q\\\mu_{\mathcal{B}}(s)=1}}\chi(s)\sum_{\substack{Q-s<r<Q\\(r,s)=1}}\left(\frac{2Q}{s} \right)^l-\sum_{\substack{s\leq Q\\\mu_{\mathcal{B}}(s)=1}}\chi(s)\sum_{\substack{Q-s<r<Q\\(r,s)=1}}\left(\left(\frac{Q+r}{s} \right)^l-\left(\frac{2Q}{s} \right)^l\right)\\&+\BigO{Q^{l-1}}=S_1+S_2+\BigO{Q^{l-1}}.\numberthis\label{highmomentsum}
\end{align*}
We begin with the first term on the right side of \eqref{highmomentsum} as
\begin{align*}
   S_1&=2^lQ^l \sum_{\substack{s\leq Q\\\mu_{\mathcal{B}}(s)=1}}\frac{\chi(s)}{s^l}\sum_{Q-s<r<Q}\sum_{d|(r,s)}\mu(d) =2^lQ^l \sum_{\substack{s\leq Q\\\mu_{\mathcal{B}}(s)=1}}\frac{\chi(s)}{s^l}\sum_{d|s}\frac{\mu(d)s}{d}  \\&=2^lQ^l\left(\sum_{s=1}^{\infty}\frac{\chi(s)\mu_{\mathcal{B}}(s)}{s^l}\sum_{d|s}\frac{\mu(d)s}{d}+\BigOkb{\frac{1}{Q}}\right)=2^lQ^l\sum_{d=1}^{\infty}\frac{\mu(d)}{d}\sum_{\substack{s=1\\d|s}}^{\infty}\frac{\chi(s)\mu_{\mathcal{B}}(s)}{s^{l-1}}+\BigOkbl{Q^{l-1}}
\\&=2^lQ^l\sum_{d=1}^{\infty}\frac{\chi(d)\mu(d)\mu_{\mathcal{B}}(d)}{d^l}\sum_{m=1}^{\infty}\frac{\chi(m)\mu(m)}{m^{l-1}}+\BigOkbl{Q^{l-1}} \\
   &=\frac{2^lQ^lL(l-1,\chi)}{L(l,\chi)}\prod_{p\in \mathcal{B}}\left(\sum_{n=0}^{\infty}\frac{\chi(p^n)}{(p^{l-1})^n} \right)^{-1}\left(\sum_{n=0}^{\infty}\frac{\chi(p^n)\mu(n)}{(p^{l})^n} \right)^{-1}+\BigOkbl{Q^{l-1}}.\numberthis\label{S1sum}
\end{align*}
For the second sum, observe that
\begin{align*}
   |S_2|&\leq \sum_{s\leq Q}\sum_{Q-s<r<Q}\frac{(2Q)^l-(Q+r)^l}{s^l} \ll_{l} \sum_{s\leq Q}\frac{Q^{l-1}}{s^{l-2}}\\
   &\ll_l\left\{\begin{array}{cc}
  Q^2\log Q,  &  \mbox{if}\  l=3,\\
   Q^{l-1},  & \mbox{if} \ l\ge 4.
\end{array}\right.\numberthis\label{S2sum}
  \end{align*}
 Inserting \eqref{S1sum} and \eqref{S2sum} in \eqref{highmomentsum}, we obtain the required result.
\end{proof}
 \section{m-correlations}
 This section contains the proofs of Theorems \ref{Theoremone} and \ref{Theoremtwo}.
\begin{proof}[Proof of Theorem \ref{Theoremone}]
 From definition \ref{def1} and \eqref{nukappa}, we have
  \begin{align*}
  S_{h_1,h_2,\cdots,h_m}(Q)  &= \sum_{\gamma_i \in \mathcal{F}_Q }{\kappa_1 \left(\frac{q_{i-1}}{Q},\frac{q_i}{Q}\right)}{\kappa_{h_1+1} \left(\frac{q_{i-1}}{Q},\frac{q_i}{Q}\right)}{\kappa_{h_2+1} \left(\frac{q_{i-1}}{Q},\frac{q_i}{Q}\right)}\cdots{\kappa_{h_m+1} \left(\frac{q_{i-1}}{Q},\frac{q_i}{Q}\right)}. 
  \end{align*} 
  Since the pairs of denominators of consecutive fractions in $\mathcal{F}_Q$ can be viewed as elements of the set \[{(a,b) \in Q\mathcal{T} \cap \mathbb{Z}^2_{vis}}=\{(a,b)\in \mathbb{Z}^2_{vis};1\le a, b\le Q\ \text{and} \ a+b>Q\},\]
  leading to
   \begin{align*}
 S_{h_1,h_2,\cdots,h_m}(Q)  &= \sum_{(a,b) \in Q\mathcal{T} \cap \mathbb{Z}^2_{vis} }{\kappa_1 \left(\frac{a}{Q},\frac{b}{Q}\right)}{\kappa_{h_1+1} \left(\frac{a}{Q},\frac{b}{Q}\right)}{\kappa_{h_2+1} \left(\frac{a}{Q},\frac{b}{Q}\right)}\cdots{\kappa_{h_m+1} \left(\frac{a}{Q},\frac{b}{Q}\right)} \\
 &= \sum_{k=1}^\infty k\sum_{(a,b) \in Q\mathcal{T}_k \cap \mathbb{Z}^2_{vis} }{\kappa_{h_1+1} \left(\frac{a}{Q},\frac{b}{Q}\right)}{\kappa_{h_2+1} \left(\frac{a}{Q},\frac{b}{Q}\right)}\cdots{\kappa_{h_m+1} \left(\frac{a}{Q},\frac{b}{Q}\right)}.
  \end{align*}
We proceed as in \cite[(2.1)]{Gologan}, 
%and note that the inner sum on the right hand side is the sum of the product of images of the map $\kappa_{h_i+1}$ of points lying in $ Q\mathcal{T}_k \cap \mathbb{Z}^2_{vis}$ times the cardinality of this set. 
and we write $\kappa_{h_i+1}$ in terms of $\kappa_1\circ T^{h_i}$ using \eqref{kappa1}, and further from \eqref{taufunction}, we have
 \iffalse
  Therefore, the inner sum becomes
  \begin{align*}
       &\sum_{(a,b) \in Q\tau_k \cap \mathbb{Z}^2_{vis} }{\kappa_{h_1+1} \left(\frac{a}{Q},\frac{b}{Q}\right)}{\kappa_{h_2+1} \left(\frac{a}{Q},\frac{b}{Q}\right)}...{\kappa_{h_m+1} \left(\frac{a}{Q},\frac{b}{Q}\right)}\\
   &= \sum_{l_1,l_2,\cdots,l_m = 1}^\infty {l_1l_2\cdots l_m \#\left\{(a,b) \in Q\mathcal{T}_k \cap \mathbb{Z}^2_{vis} ; \kappa_{h_1+1} \left(\frac{a}{Q},\frac{b}{Q}\right) = l_1 ,\cdots, \kappa_{h_m+1}\left(\frac{a}{Q},\frac{b}{Q}\right) =l_m \right \}}\\
  &=\sum_{l_1,l_2,\cdots,l_m = 1}^\infty {l_1l_2...l_m \#\left\{(a,b) \in Q\mathcal{T}_k \cap \mathbb{Z}^2_{vis} ; T^{h_1} \left(\frac{a}{Q},\frac{b}{Q}\right) \in \mathcal{T}_{l_1} ,\cdots, T^{h_m}\left(\frac{a}{Q},\frac{b}{Q}\right) \in \mathcal{T}_{l_m} \right \}}\\ 
  &=\sum_{l_1,l_2,\cdots,l_m = 1}^\infty {l_1l_2\cdots l_m \#\left\{(a,b) \in Q\mathcal{T}_k \cap \mathbb{Z}^2_{vis} ; (a,b)\in Q(T^{-h_1}\mathcal{T}_{l_1}\cap T^{-h_2}\mathcal{T}_{l_2}\cap \cdots \cap T^{-h_m}\mathcal{T}_{l_m})   \right \}}\\
  & =\sum_{l_1,l_2,\cdots,l_m = 1}^\infty {l_1l_2\cdots l_m \#\left\{ (Q(\mathcal{T}_k \cap T^{-h_1}\mathcal{T}_{l_1}\cap T^{-h_2}\mathcal{T}_{l_2}\cap \cdots \cap T^{-h_m}\mathcal{T}_{l_m})\cap \mathbb{Z}^2_{vis})   \right \}.}
  \end{align*}
  And so \fi
   \[S_{h_1,h_2,\cdots,h_m}(Q) = \sum_{k,l_1,l_2,\cdots,l_m = 1}^\infty {kl_1l_2\cdots l_m \#  (Q(\mathcal{T}_k \cap T^{-h_1}\mathcal{T}_{l_1}\cap T^{-h_2}\mathcal{T}_{l_2}\cap \cdots \cap T^{-h_m}\mathcal{T}_{l_m})\cap \mathbb{Z}^2_{vis})  },\]%\numberthis\label{thm11} 
 where $l_i:=\kappa_{h_i+1} \left(\dfrac{a}{Q},\dfrac{b}{Q}\right)$ for all $1\le i\le m.$ The maps $\mathcal{T}_j$ and $T$ as in \eqref{taufunction} and \eqref{tfunction}, respectively. 
 Next, we write
 \[A_{k,l_1,l_2,\cdots,l_m}(Q) := \# (Q(\mathcal{T}^*_k \cap T^{-h_1}\mathcal{T}^*_{l_1}\cap T^{-h_2}\mathcal{T}^*_{l_2}\cap \cdots \cap T^{-h_m}\mathcal{T}^*_{l_m})\cap \mathbb{Z}^2_{vis}),\] %\numberthis\label{thm1one} 
 where
$\mathcal{T}^*_{k} = \bigcup_{n=k}^{\infty} \mathcal{T}_n.$
  We observe that
  \begin{align*}
     &\mathcal{T}^*_k \cap T^{-h_1}\mathcal{T}^*_{l_1}\cap T^{-h_2}\mathcal{T}^*_{l_2}\cap \cdots \cap T^{-h_m}\mathcal{T}^*_{l_m}\\& = \bigcup_{p=k}^{\infty} \bigcup_{n_1=l_1}^{\infty}\bigcup_{n_2=l_2}^{\infty}\cdots\bigcup_{n_m=l_m}^{\infty}(\mathcal{T}_p \cap T^{-h_1}\mathcal{T}_{n_1}\cap T^{-h_2}\mathcal{T}_{n_2}\cap \cdots \cap T^{-h_m}\mathcal{T}_{n_m}). \numberthis\label{thm1two}
 \end{align*}
  Since the sets on the right hand side of the above equation are mutually disjoint, we have
 \begin{align*}
 \sum_{k,l_1,l_2,\cdots,l_m = 1}^{\infty} A_{k,l_1,l_2,\cdots,l_m}(Q) &= \sum_{p,n_1,n_2,\cdots,n_m = 1}^{\infty} p n_1n_2\cdots n_m \\
& \times \#  (Q(\mathcal{T}_p \cap T^{-h_1}\mathcal{T}_{n_1}\cap T^{-h_2}\mathcal{T}_{n_2}\cap  \cdots \cap T^{-h_m}\mathcal{T}_{n_m})\cap \mathbb{Z}^2_{vis})\\&= S_{h_1,h_2,\cdots,h_m}(Q). \numberthis\label{thm1three}
 \end{align*}
% By substituting  \eqref{thm1three}  in \eqref{thm11}, we have
 % \[S_{h_1,h_2,\cdots,h_m}(Q) = \sum_{k,l_1,l_2,\cdots,l_m = 1}^\infty A_{k,l_1,l_2,\cdots,l_m}(Q). \numberthis\label{thm14}\]
Since $ Q\mathcal{T}^*_{k} \cap \mathbb{Z}^2_{vis} = \emptyset$ if $k>2Q,$ so \eqref{note1} implies that $A_{k,l_1,l_2,...,l_m}(Q) = 0$, if $\max (k, l_1, l_2,\cdots,l_m) > 2Q.$ From Lemma \ref{Corollaryone}, $A_{k,l_1,l_2,\cdots,l_m}(Q) = 0$ if $\min (l_i,l_j) > c_h.$ 
Hence \eqref{thm1three} gives
  \begin{align*}
 S_{h_1,h_2,\cdots,h_m}(Q) &= \sum_{k,l_1,l_2,\cdots,l_m = 1}^{2Q} A_{k,l_1,l_2,\cdots,l_m}(Q) \\
 =& \left(\sum_{k = c_h+1}^{2Q} \sum_{l_1,l_2,\cdots,l_m=1}^{c_h}+ \sum_{l_1 = c_h+1}^{2Q} \sum_{k,l_2,\cdots,l_m=1}^{c_h}+\cdots+ \sum_{l_m = 1}^{2Q} \sum_{k,l_1,l_2,\cdots,l_{m-1}=1}^{c_h}\right)A_{k,l_1,l_2,\cdots,l_m}(Q)\\
 &= M_0 + M_1 + \cdots+ M_{m}. \numberthis\label{thm15}
 \end{align*}
 \iffalse
 where
 \[M_1 = \sum_{k = c_h+1}^{2Q} \sum_{l_1,l_2,\cdots,l_m=1}^{c_h}A_{k,l_1,l_2,\cdots,l_m}(Q), \numberthis\label{thm16}\]
 \[M_2 = \sum_{l_1 = c_h+1}^{2Q} \sum_{k,l_2,\cdots,l_m=1}^{c_h}A_{k,l_1,l_2,\cdots,l_m}(Q), \numberthis\label{thm162}\]
 \[ \vdots\]
 \[M_m = \sum_{l_{m-1}= c_h+1}^{2Q} \sum_{k,l_1,l_2,\cdots,l_{m-2},l_m=1}^{c_h}A_{k,l_1,l_2,\cdots,l_m}(Q), \numberthis\label{thm17}\]
 \[M_{m+1} = \sum_{l_m = 1}^{2Q} \sum_{k,l_1,l_2,\cdots,l_{m-1}=1}^{c_h}A_{k,l_1,l_2,\cdots,l_m}(Q). \numberthis\label{thm18}\]\fi
We now estimate $M_{i}$ for $\ 1\leq i\leq m$. By \eqref{note1}, we have $A_{k,l_1,l_2,\cdots,l_m}(Q)= Q(T^{h_i}\mathcal{T}^*_k \cap T^{h_i-h_1}\mathcal{T}^*_{l_1}\cap T^{h_i-h_2}\mathcal{T}^*_{l_2}\cap \cdots \cap \mathcal{T}^*_{l_i}\cap\cdots\cap T^{h_i-h_m}\mathcal{T}^*_{l_m})$.  Let
 \[\Omega_{i} := Q(T^{h_i}\mathcal{T}^*_k \cap T^{h_i-h_1}\mathcal{T}^*_{l_1}\cap T^{h_i-h_2}\mathcal{T}^*_{l_2}\cap \cdots \cap \mathcal{T}^*_{l_i}\cap\cdots\cap T^{h_i-h_m}\mathcal{T}^*_{l_m}). \]
  Since $l_i>c_h$ and $k, l_j \leq c_h$ for all $j\ne i, 1\leq j\leq m,$ and $\mathcal{T}^*_k$ is the union of Farey triangles with the property that $\mathcal{T}^*_k \subseteq \mathcal{T}^*_l$ if $k>l$, we have
 %\[\mathcal{T}^*_{k} \subseteq \mathcal{T}^*_{l_1}, \mathcal{T}^*_{k} \subseteq \mathcal{T}^*_{l_2},\cdots, \mathcal{T}^*_{k} \subseteq \mathcal{T}^*_{l_m}. \numberthis \label{thm19}\]
 \[\mathcal{T}^*_{l_i} \subseteq \mathcal{T}^*_k\cap_{\substack{j=1\\j\ne i}}^{m} \mathcal{T}^*_{l_j}.\]
Now, the above equation and the fact that $T$ is area preserving implies that
  \begin{center}
     area$(\Omega_{i})\leq \text{area}(Q\mathcal{T}_{l_i}^*) = \BigO{\frac{Q^2}{l_i^2}} $, and length$ (\partial \Omega_{i} ) =\BigOhh{\frac{Q}{l_i}}$.
\end{center}
Applying Lemma \ref{lemma1} to $\Omega_{i+1}$, we get
\begin{align*}
A_{k,l_1,l_2,\cdots,l_m}(Q) &= \frac{6Q^2}{\pi^2}\ \text{area}(T^{h_i}\mathcal{T}^*_k \cap T^{h_i-h_1}\mathcal{T}^*_{l_1}\cap T^{h_i-h_2}\mathcal{T}^*_{l_2}\cap \cdots \cap \mathcal{T}^*_{l_i}\cap\cdots\cap T^{h_i-h_m}\mathcal{T}^*_{l_m})\\
&+ \BigOhh{\frac{Q \log Q}{l_i}}.
\end{align*}
Hence, we have
\begin{align*}
    M_{i} &= \frac{6Q^2}{\pi ^2}\sum_{l_i = c_h+1}^{2Q} \sum_{\substack{k,l_j=1\\ 1\leq j\leq m,\ j\ne i}}^{c_h} \text{area}(T^{h_i}\mathcal{T}^*_k \cap T^{h_i-h_1}\mathcal{T}^*_{l_1}\cap T^{h_i-h_2}\mathcal{T}^*_{l_2}\cap \cdots \cap \mathcal{T}^*_{l_i}\cap\cdots\cap T^{h_i-h_m}\mathcal{T}^*_{l_m}) \\
    &+ \BigOhh{Q \log Q \sum_{l_i = c_h+1}^{2Q} \sum_{\substack{k,l_j=1\\ 1\leq j\leq m,\ j\ne i}}^{c_h} \frac{1}{l_i}}.
     \end{align*}
     From Lemma \ref{Corollaryone}, for $k>2Q$ and for any $l_i>c_h$, $\mathcal{T}^*_k \cap T^{-h_1}\mathcal{T}^*_{l_1}\cap T^{-h_2}\mathcal{T}^*_{l_2}\cap \cdots \cap T^{-h_m}\mathcal{T}^*_{l_m})=\emptyset$ and using the fact that $T$ is area preserving, we have
     \begin{align*}
 M_{i}  & = \frac{6Q^2}{\pi ^2}\sum_{l_i = c_h+1}^{\infty} \sum_{\substack{k,l_j=1\\ 1\leq j\leq m,\ j\ne i}}^{\infty} \text{area}(\mathcal{T}^*_k \cap T^{-h_1}\mathcal{T}^*_{l_1}\cap T^{-h_2}\mathcal{T}^*_{l_2}\cap \cdots \cap T^{-h_m}\mathcal{T}^*_{l_m}) \\
  &+ \BigOhh{Q \log^2 Q}.\numberthis\label{thm110}
 \end{align*}
 We proceed in the similar manner to estimate $M_0$ with $\Omega_0=Q(\mathcal{T}^*_k \cap T^{-h_1}\mathcal{T}^*_{l_1}\cap T^{-h_2}\mathcal{T}^*_{l_2}\cap \cdots \cap T^{-h_m}\mathcal{T}^*_{l_m})$ and get  
 \begin{align*}
    M_0 & = \frac{6Q^2}{\pi ^2}\sum_{k = c_h+1}^{\infty} \sum_{l_1,l_2,\cdots,l_m=1}^{\infty}\text{area}(\mathcal{T}^*_k \cap T^{-h_1}\mathcal{T}^*_{l_1}\cap T^{-h_2}\mathcal{T}^*_{l_2}\cap \cdots \cap T^{-h_m}\mathcal{T}^*_{l_m}) \\
&+\BigOhh{Q \log^2 Q }.\numberthis \label{thm112}
\end{align*}
\iffalse
\vdots
 \[M_{m+1} = \frac{6Q^2}{\pi ^2}\sum_{k,l_1,\cdots,l_{m-1} = 1}^{c_h} \sum_{l_m=1}^{\infty}\text{area}(\mathcal{T}^*_k \cap T^{-h_1}\mathcal{T}^*_{l_1}\cap T^{-h_2}\mathcal{T}^*_{l_2}\cap \cdots \cap T^{-h_m}\mathcal{T}^*_{l_m}) +\BigOhh{Q \log^2 Q }.\numberthis\label{thm1132} \] \fi
The statement of the theorem now follows by using equations \eqref{thm15}, \eqref{thm110}, and \eqref{thm112} as below.
\begin{align*}
S_{h_1,h_2,\cdots,h_m}(Q) =& \frac{6Q^2}{\pi^2}\sum_{k,l_1,l_2,\cdots,l_m = 1}^{\infty} \text{area}(\mathcal{T}^*_k \cap T^{-h_1}\mathcal{T}^*_{l_1}\cap T^{-h_2}\mathcal{T}^*_{l_2}\cap \cdots \cap T^{-h_m}\mathcal{T}^*_{l_m}) \\
&+ \BigOhh{Q \log^2 Q }.
%\\=&\frac{6Q^2}{\pi^2}\sum_{k,l_1,l_2,\cdots,l_m = 1}^{\infty} kl_1l_2\cdots l_m \ \text{area}(\mathcal{T}_k \cap T^{-h_1}\mathcal{T}_{l_1}\cap T^{-h_2}\mathcal{T}_{l_2}\cap \cdots \cap T^{-h_m}\mathcal{T}_{l_m})\\&+ \BigOhh{Q \log^2 Q }. %\\=&\frac{6Q^2}{\pi^2}\frac{A(h_1,h_2,\cdots,h_m)}{2}+\BigOhh{Q \log^2 Q },
\numberthis\label{thm114}
\end{align*}
%where the last stile \textcolor{red}{?} follows from \eqref{thm1two}.
Now from \eqref{thm1two}, the right side of \eqref{thm114} is equal to
\begin{align*}
&\sum_{k,l_1,l_2,\cdots,l_m = 1}^{\infty} \text{area}(\mathcal{T}^*_k \cap T^{-h_1}\mathcal{T}^*_{l_1}\cap T^{-h_2}\mathcal{T}^*_{l_2}\cap \cdots \cap T^{-h_m}\mathcal{T}^*_{l_m}) \\
&= \sum_{k,l_1,l_2,\cdots,l_m = 1}^{\infty} kl_1l_2\cdots l_m \ \text{area}(\mathcal{T}_k \cap T^{-h_1}\mathcal{T}_{l_1}\cap T^{-h_2}\mathcal{T}_{l_2}\cap \cdots \cap T^{-h_m}\mathcal{T}_{l_m})\\
&=: \frac{A(h_1,h_2,\cdots,h_m)}{2}, \numberthis\label{thm115}
\end{align*}
which yields the required result.
\end{proof}
\begin{proof}[Proof of Theorem \ref{Theoremtwo}]
 To prove this,
we use definition \ref{def2}, Remark \ref{subinterval} and proceed as Theorem \ref{Theoremone}, to obtain 
\begin{align*}
S_{h_1,h_2,\cdots,h_m,t}(Q) &= \sum_{\gamma_{i}\in \mathcal{F}_Q\cap [0,t]} {\kappa_1 \left(\frac{q_{i-1}}{Q},\frac{q_i}{Q}\right)}{\kappa_{h_1+1} \left(\frac{q_{i-1}}{Q},\frac{q_i}{Q}\right)}{\kappa_{h_2+1} \left(\frac{q_{i-1}}{Q},\frac{q_i}{Q}\right)}\cdots{\kappa_{h_m+1} \left(\frac{q_{i-1}}{Q},\frac{q_i}{Q}\right)} \\
&= \sum_{\substack{(a,b)\in Q\mathcal{T} \\ b^\prime  \in I_a}} {\kappa_1 \left(\frac{a}{Q},\frac{b}{Q}\right)}{\kappa_{h_1+1} \left(\frac{a}{Q},\frac{b}{Q}\right)}{\kappa_{h_2+1} \left(\frac{a}{Q},\frac{b}{Q}\right)}...{\kappa_{h_m+1} \left(\frac{a}{Q},\frac{b}{Q}\right)} \\
&= \sum_{k,l_1,l_2,\cdots,l_m}^{\infty} A_{k,l_1,l_2,\cdots,l_m,t}(Q), \numberthis\label{thm21}
\end{align*}
where
$A_{k,l_1,l_2,\cdots,l_m,t}(Q) = \# (Q(\mathcal{T}_k^*\cap T^{-h_1}\mathcal{T}_{l_1}^* \cap \cdots\cap T^{-h_m}\mathcal{T}_{l_m}^*)\cap \{(a,b)|b^\prime \in I_a\})$ {and} $I = (0,t].$
Applying \cite[Lemma ~10]{Boca} to $f(a,b)=1$, $\Omega=Q(\mathcal{T}_{k}^{*}\cap T^{-h_1}\mathcal{T}_{l_1}^*\cap \cdots \cap T^{-h_m}\mathcal{T}_{l_m}^*)$, $A=\BigO{Q}$, $R_1=R_2=\BigO{Q/k}$, we have 
\begin{align*}
    A_{k,l_1,l_2,\cdots,l_m,t}(Q) &= t A_{k,l_1,l_2,\cdots,l_m}(Q) + \BigOhh{\frac{Q}{k^2} + \frac{Q}{k}\left(Q + \frac{Q}{k}\right)^{1/2+\epsilon}}\\
&= t A_{k,l_1,l_2,\cdots,l_m}(Q) + \BigOhhe{\frac{Q^{3/2 + \epsilon}}{k}}.\end{align*}
Substituting the above in \eqref{thm21}, we obtain the statement of the theorem. 
\end{proof}
\section{On the constant $A(h_1,h_2,\cdots,h_m)$ }
%This section is devoted to $A(h_1,h_2,\cdots,h_m)$ in \eqref{thm115}. 
In this section, we first find a bound for $A(h_1,h_2,\cdots,h_m)$, and then prove that $A(h_1,h_2,\cdots,h_m)$ is a rational number.  
 We follow  \cite{Gologan} and define a function $f$ as
\[f := f_0 + R_0,\numberthis\label{functionf}\]
where
\[f_0 = \sum_{n=1}^{c_h} e_{\mathcal{T}^*_n} ,\hspace{3mm} R_0 = \sum_{n=c_h+1}^{\infty} e_{\mathcal{T}^*_n},\numberthis\label{est1}\]
where $e_A$ denotes the characteristic function of the set $A.$ 
This implies
\[f(f\circ T^{h_1})(f\circ T^{h_2})\cdots(f\circ T^{h_m})
= \sum_{k,l_1,l_2,\cdots,l_m=1}^{\infty} e_{\mathcal{T}^*_k \cap T^{-h_1}\mathcal{T}^*_{l_1}\cap T^{-h_2}\mathcal{T}^*_{l_2}\cap \cdots \cap T^{-h_m}\mathcal{T}^*_{l_m}}. \numberthis\label{thm116}\] Using \eqref{thm115} and \eqref{thm116}, we write
\[A(h_1,h_2,...,h_m) = 2 \iint_\mathcal{T}f(s,t)(f\circ T^{h_1})(s,t)(f\circ T^{h_2})(s,t)\cdots(f\circ T^{h_m})(s,t)\ ds\ dt \numberthis\label{thm117}.\]
%From \eqref{thm116}, we have
%\[f(f\circ T^{-h_1})(f\circ T^{-h_2})\cdots(f\circ T^{-h_m})
%= \sum_{k,l_1,l_2,\cdots,l_m=1}^{\infty} e_{\mathcal{T}^*_k \cap T^{-h_1}\mathcal{T}^*_{l_1}\cap T^{-h_2}\mathcal{T}^*_{l_2}\cap \cdots \cap T^{-h_m}\mathcal{T}^*_{l_m}}.\]
Using \eqref{functionf}, \eqref{est1}, and \eqref{note3}, we have
\[f(R_0\circ T^{h_1})(f_0\circ T^{h_2})\cdots(f_0\circ T^{h_m}) = \sum_{k=1}^{2}\sum_{l_1=c_h+1}^{\infty}\sum_{l_2,\cdots,l_m = 1}^{c_h} e_{\mathcal{T}^*_k \cap T^{-h_1}\mathcal{T}^*_{l_1}\cap T^{-h_2}\mathcal{T}^*_{l_2}\cap \cdots \cap T^{-h_m}\mathcal{T}^*_{l_m}}.\] 
Therefore,
\begin{align*}
&\iint_\mathcal{T}f(s,t)(R_0\circ T^{h_1})(s,t)(f_0\circ T^{h_2})(s,t)\cdots(f_0\circ T^{h_m})(s,t)\ ds\ dt \\
&=\sum_{k=1}^{2}\sum_{l_1=c_h+1}^{\infty}\sum_{l_2,\cdots,l_m = 1}^{c_h} \text{area}(\mathcal{T}^*_k \cap T^{-h_1}\mathcal{T}^*_{l_1}\cap T^{-h_2}\mathcal{T}^*_{l_2}\cap \cdots \cap T^{-h_m}\mathcal{T}^*_{l_m}) \\
&\leq 2\sum_{l_1=c_h+1}^{\infty}\sum_{l_2,\cdots,l_m = 1}^{c_h} \text{area}( T^{-h_1}\mathcal{T}^*_{l_1}\cap T^{-h_2}\mathcal{T}^*_{l_2}\cap \cdots \cap T^{-h_m}\mathcal{T}^*_{l_m})\leq 2^m  \sum_{l_1=c_h}^{\infty}\frac{1}{l_1^2}\ll_m 1.
\end{align*}
Similarly, we have
\[\iint_\mathcal{T}f(s,t)(f_0\circ T^{h_1})(s,t)(R_0\circ T^{h_2})(s,t)\cdots(f_0\circ T^{h_m})(s,t)\ ds\ dt \ll_m 1,\] and
\[\iint_\mathcal{T}f(s,t)(f_0\circ T^{h_1})(s,t)(f_0\circ T^{h_2})(s,t)\cdots(R_0\circ T^{h_m})(s,t)\ ds\ dt \ll_m 1.\] Using Lemma \ref{Corollaryone}, we have
\[f(R_0\circ T^{h_1})(R_0\circ T^{h_2})(f_0\circ T^{h_3})\cdots(f_0\circ T^{h_m}) = 0.\]
Also by Lemma \ref{Corollaryone},  one observes that the product $f(R_0\circ T^{h_1})(R_0\circ T^{h_2})(f_0\circ T^{h_3})\cdots(f_0\circ T^{h_m})$  contributes zero if it has two or more of the functions $R_0\circ T^{h_i}.$
Hence \eqref{thm117} implies that
{
\[A(h_1,h_2,\cdots,h_m) \ll_m c_h + 2 \iint_\mathcal{T}f_0(s,t)(f_0\circ T^{h_1})(s,t)(f_0\circ T^{h_2})(s,t)\cdots(f_0\circ T^{h_m})(s,t)\ ds\ dt.\numberthis\label{thm118} \]}
Using \eqref{thm116}, we obtain
\[|| (f_0\circ T^{h_1})(s,t)(f_0\circ T^{h_2})(s,t)\cdots(f_0\circ T^{h_m})(s,t) || \ll_m c_h. \numberthis\label{thm120}\]
Observe that
\begin{align*}
 f_0&=\sum_{n=1}^{c_h}e_{\tau_{k^*}}=\sum_{n=1}^{c_h}\sum_{k=m}^{\infty}e_{\tau_k} =\sum_{k=1}^{\infty}\sum_{n=1}^{\min (k,c_h)}e_{\tau_k}=\sum_{k=1}^{\infty}\min(k,c_h)e_{\tau_k}=\sum_{k=1}^{c_h-1}ke_{\tau_k}+c_h\sum_{k=c_h}^{\infty}e_{\tau_k},
\end{align*}
which gives
\[|| f_0 ||_{L^2(\mathcal{T})} \ll 1+ \log h. \numberthis\label{thm119}\]
Applying Cauchy-Schwarz inequality on \eqref{thm118}, and then using \eqref{thm119} and \eqref{thm120}, we get
$A(h_1,h_2,\cdots,h_m) \ll_m   (1+ \log h)c_h. $
Next,  we prove that $A(h_1,h_2,\cdots,h_m)$ is rational for all $h_i\in \mathbb{N}$ and $1\le i\le m$. Using Lemma \ref{Corollaryone}, we write
\begin{align}\label{Ah1h2}
    A(h_1,h_2,\cdots,h_m) %=&\sum_{k,l_1,l_2,\cdots,l_m=1}^\infty kl_1l_2\cdots l_m \ \text{area}(\mathcal{T}_k \cap T^{-h_1}\mathcal{T}_{l_1}\cap \cdots \cap T^{-h_m}\mathcal{T}_{l_m})\\=&
   = 2\sum_{i=0}^{m+1}A_i,
    %
   % \\
 %   =&\left(\sum_{k,l_1,l_2,\cdots,l_m=1}^{c_h}+\sum_{l_1,l_2,\cdots,l_{m}=1}^{c_h}\sum_{k=c_h+1}^\infty+\cdots+  \sum_{k,l_1,l_2,\cdots,l_{m-1}=1}^{c_h}\sum_{l_m=c_h+1}^\infty\right)\\ &\times kl_1l_2\cdots l_m \ \text{area}(\mathcal{T}_k \cap T^{-h_1}\mathcal{T}_{l_1}\cap \cdots \cap T^{-h_m}\mathcal{T}_{l_m}) \\
  %  =&A_0+A_1+\cdots+A_{m+1},
    \end{align}
    where
    \[A_0=\sum_{\substack{l_j=1\\ 1\leq j \leq m+1}}^{c_h} l_1l_2\cdots l_ml_{m+1} \ \text{area}(T^{-h_1}\mathcal{T}_{l_1}\cap \cdots \cap T^{-h_m}\mathcal{T}_{l_m}\cap \mathcal{T}_{l_{m+1}}),\]and
    \begin{align*}
     A_{i} &=\sum_{\substack{l_j=1\\ 1\leq j\leq m+1,\ j\ne i}}^{c_h}\sum_{l_i=c_h+1}^\infty l_1l_2\cdots l_ml_{m+1}\ \text{area}( T^{-h_1}\mathcal{T}_{l_1}\cap\cdots \cap T^{-h_i}\mathcal{T}_{l_i}\cap\cdots \cap T^{-h_m}\mathcal{T}_{l_m}\cap \mathcal{T}_{l_{m+1}}).
     \end{align*}
    %\textcolor{red}{define $A_0$ and $A_1$}
    \iffalse
    where
    \[A_0=\sum_{k,l_1,l_2,\cdots,l_m=1}^{c_h} kl_1l_2\cdots l_m \ \text{area}(\mathcal{T}_k \cap T^{-h_1}\mathcal{T}_{l_1}\cap \cdots \cap T^{-h_m}\mathcal{T}_{l_m}),\]
  \[  A_1=\sum_{l_1,l_2,\cdots,l_{m}=1}^{c_h}\sum_{k=c_h+1}^\infty kl_1l_2...l_m \ \text{area}(\mathcal{T}_k \cap T^{-h_1}\mathcal{T}_{l_1}\cap \cdots \cap T^{-h_m}\mathcal{T}_{l_m}),\]
    \[A_2= \sum_{k,l_2,\cdots,l_{m}=1}^{c_h}\sum_{l_1=c_h+1}^\infty kl_1l_2\cdots l_m\  \text{area}(\mathcal{T}_k \cap T^{-h_1}\mathcal{T}_{l_1}\cap \cdots \cap T^{-h_m}\mathcal{T}_{l_m}),\]
     \[\vdots\]
    \[A_{m+1}= \sum_{k,l_1,l_2,\cdots,l_{m-1}=1}^{c_h}\sum_{l_m=c_h+1}^\infty kl_1l_2\cdots l_m\  \text{area}(\mathcal{T}_k \cap T^{-h_1}\mathcal{T}_{l_1}\cap \cdots \cap T^{-h_m}\mathcal{T}_{l_m}). \]\fi
    \\
Case (i): If $h_j\geq 2\ \text{for\ all\ } j $. Using \eqref{note3} and the fact that $T$ is area preserving, we obtain for all $1\leq i \leq m+1$ 
\begin{align*}
 A_{i}&=\sum_{\substack{l_j=1\\ 1\leq j\leq m,\ j\ne i}}^{c_h}\sum_{l_i=c_h+1}^\infty 2l_1l_2\cdots l_m \ \text{area}(T^{-h_1}\mathcal{T}_{l_1}\cap T^{-h_2}\mathcal{T}_{l_2}\cap\cdots \cap T^{-h_i}\mathcal{T}_{l_i}\cap\cdots \cap T^{-h_m}\mathcal{T}_{l_m})\\
 &= 2\sum_{\substack{l_j=1\\ 1\leq j\leq m,\ j\ne i}}^{c_h}\sum_{l_i=c_h+1}^\infty l_1l_2l_3\cdots l_m \ \text{area}(\mathcal{T}_{l_1}\cap T^{h_1-h_2}\mathcal{T}_{l_2}\cap \cdots \cap T^{h_1-h_i}\mathcal{T}_{l_i}\cap\cdots T^{h_1-h_m}\mathcal{T}_{l_m})\\
&= 2^2\sum_{\substack{l_j=1\\ 2\leq j\leq m,\ j\ne i}}^{c_h}\sum_{l_i=c_h+1}^\infty l_2l_3...l_m \text{area}(T^{h_1-h_2}\mathcal{T}_{l_2}\cap \cdots \cap T^{h_1-h_i}\mathcal{T}_{l_i}\cap\cdots T^{h_1-h_m}\mathcal{T}_{l_m})\\
   %&= 2^2\sum_{\substack{l_j=1\\ 2\leq j\leq m,\ j\ne i}}^{c_h}\sum_{l_i=c_h+1}^\infty l_2l_3\cdots l_m \text{area}(\mathcal{T}_{l_2}\cap \cdots \cap T^{h_1-h_i}\mathcal{T}_{l_i}\cap\cdots T^{h_1-h_m}\mathcal{T}_{l_m})\\
 &=2^{m}\sum_{l_i=c_h+1}^\infty l_i \text{area}(\mathcal{T}_{l_i})
    = \frac{2^{m+2}}{c_h+2}.
\end{align*}
Therefore, $A_{i}\in\mathbb{Q}$ for all $1\leq i \leq m+1$. Since each region $T^h\mathcal{T}_m$ is a finite union of triangles with rational vertex coordinates, $A_0\in\mathbb{Q}$. This together with \eqref{Ah1h2} implies $A(h_1,h_2,...,h_m)\in\mathbb{Q}$ for all $h_i \geq 2.$
\\
Case (ii): If $h_j=1$ for all $j$, then using \eqref{note4}, and the fact that $T$ is area preserving, we get
\begin{align*}
   A_{m+1} &=\sum_{\substack{l_j=1\\ 1\leq j\leq m,}}^{c_h} \sum_{l_m=c_h+1}^{\infty} l_1l_2\cdots l_ml_{m+1} \hspace{1mm} \text{area}( \mathcal{T}_{l_1}\cap \cdots \cap \mathcal{T}_{l_m}\cap T\mathcal{T}_{l_{m+1}})  \\
    &= \sum_{l_m=c_h+1}^{\infty} l_i \hspace{1mm}\ \text{area}(T\mathcal{T}_{l_m}) 
    = \frac{4}{c_h+2}.
\end{align*}
so $A_{m+1}$ is rational. And similar to above, one obtains $A_1\in\mathbb{Q}$. Note that $A_i=0$ for all $1\leq i\leq m$, since  $ \mathcal{T}_{l_1}\cap \cdots \cap \mathcal{T}_{l_m}\cap T\mathcal{T}_{l_{m+1}}=\emptyset$ if $c_h< l_i \leq \infty, i\ne m+1$ and $1\leq l_j \leq c_h$ for $j\ne i$.
\\ Case (iii): Rationality of $A(h_1,h_2,\cdots ,h_m)$, when $h_j=1$ for some $j$'s and $h_i\geq 2$ for $i\neq j$, is obtained in a similar way.

\section{Appendix}\label{section8}
As mentioned in the introduction, the case for squarefree denominators demands further investigation. Since characteristic function of square free numbers $\mu(.)^2$ is not completely multiplicative, the authors in \cite{MR2424917} miss the coprimality condition when they substitute the divisor condition by a hyperbolic area and factor $\mu(.)^2$. 
%then to break $\mu(.)^2$, there should be an extra condition on gcd.
For example, the second last equation on page 135 of \cite{MR2424917} should be 
\[\sum_{\substack{s\le Q\\\mu(s)^2=1}}\chi(s)\sum_{d|s}\mu(d)\left[\frac{Q}{d}\right]=Q\sum_{d\le Q}\frac{\chi(d)\mu(d)}{d}\sum_{\substack{m\le Q/d\\\gcd(m,d)=1} }\chi(m){\mu(m)^2}+O(Q\log Q).\]
In this section, we rectify all these errors and give new estimates on the moments of index function over squarefree denominators. 
\iffalse For the set $S:=\{n\in \mathbb{Z}: \forall{k}\in\mathbb{Z}_{\ge 1}: k^2\not| N\}$ of square free numbers, then characteristic function for square free numbers as
\begin{align*}
    \mu(n)^2:=\left\{\begin{array}{cc}
       1,  &  \mbox{if}\ n\in S,\\
       0,  & \mbox{otherwise}.
    \end{array}\right.
\end{align*}
\fi
\begin{thm}\label{t1}
 Fix a positive integer $k$, and a Dirichlet character $\chi$ modulo k. Then, for all large positive integers $Q,$ we have 
 \begin{align*}
        \mathcal{M}_{1,\Box}(\chi,Q)&=\left\{\begin{array}{cc}
         \BigOk
{Q^{\frac{3}{2}}(\log Q)},      &  \mbox{if}\ \chi\ne\chi_0\\ \\  
       \dfrac{12Q^2}{\pi^2}\prod_{p|k}\dfrac{p^2}{p^2+p-1}\prod_{p}\left(1-\dfrac{1}{p(p+1)} \right)\\-\dfrac{Q^2}{2k L(2,\chi_0)}\prod_{p\nmid k}\left(1-\dfrac{\phi(k)}{p(p+1)} \right)+\BigOk
{Q^{\frac{3}{2}}},      &  \mbox{if}\ \chi=\chi_0.
       \end{array}\right.
       \end{align*}
\end{thm}
\begin{proof}
%Using \eqref{41}, we have 
We follow \eqref{51} with $\mu(s)^2$ which gives
\begin{align*}
   % \mathcal{M}_{1,\Box}(\chi,Q)&=\sum_{\substack{\gamma_i=\frac{b}{s}\in \mathcal{F}_Q\\ \mu(s)^2=1 }}\chi(s)\nu(\gamma_i)=\sum_{\substack{s\leq Q \\\mu(s)^2=1 }}\chi(s)T(s) \\
&\mathcal{M}_{1,\Box}(\chi,Q)=2\sum_{\substack{s\leq Q\\ \mu(s)^2=1}}\chi(s)\sum_{d|s}\mu(d)\left\lfloor{\frac{Q}{d}}\right\rfloor - \sum_{\substack{s\leq Q\\\mu(s)^2=1}}\chi(s)\phi(s) +1=S_1+S_2+1. \numberthis\label{s1}
\end{align*}
%We examine the first sum on the far right side of \eqref{s1} as 
For $\chi=\chi_0,$ we have \cite[see (2.13)]{MR2424917},
\[S_1=\frac{6Q^2}{\pi^2}\prod_{p|k}\left(1-\frac{1}{p+1}\right)\prod_{p\nmid k}\left(1-\frac{1}{p(p+1)} \right)+\BigO{Q^{\frac{3}{2}}}.\numberthis\label{s21} \]
For $\chi\ne \chi_0,$
by \eqref{52}, 
\begin{align}
 %&=\sum_{\substack{s\leq Q\\ \mu(s)^2=1}} \sum_{d|s} \mu(d) \left(\frac{Q}{d}+\BigO{1}\right) \\
 % &=Q\sum_{d\leq Q}\frac{\mu(d)}{d}\sum_{\substack{s\leq Q \\ d|s \\ \mu(s)^2=1}}\chi(s) +\BigO{\sum_{s\leq Q}\tau(s)} \\
S_1  &= Q\sum_{d\leq Q}\frac{\mu(d)}{d}\sum_{\substack{s\leq Q \\ d|s }}\chi(s) \mu(s)^2 +\BigO{Q \text{log}Q}
=  Q\sum_{d\leq Q}\frac{\chi(d)\mu(d)}{d}\sum_{\substack{m\leq \frac{Q}{d}\\ (m,d)=1}}\chi(m) \mu(m)^2 +\BigO{Q \text{log}Q}. \label{s2}\end{align} 
The inner sum in \eqref{s2} is estimated as
\begin{align*}
S_{11}:=\sum_{\substack{m\leq \frac{Q}{d}\\ (m,d)=1}}\chi(m) \mu(m)^2&=\sum_{\substack{m\leq \frac{Q}{d}\\ (m,d)=1}}\sum_{r^2|m}\chi(m)\mu(r)=\sum_{r^2\leq \frac{Q}{d}}\mu(r)\sum_{\substack{m\leq \frac{Q}{d}\\r^2|m\\(m,d)=1}}\chi(m) \\
&=\sum_{\substack{r^2\leq \frac{Q}{d}\\(r^2,d)=1}}\chi(r^2)\mu(r)\sum_{\substack{t\leq \frac{Q}{dr^2}\\(t,d)=1}}\chi(t).\numberthis\label{s3}
\end{align*}
 We employ Perron's formula to estimate the inner sum in \eqref{s3}.  The Dirichlet series is given by
 \begin{align*}
     F(s)&=\sum_{\substack{n=1\\(n,d)=1}}^{\infty}\frac{\chi(n)}{n^s}=L(s,\chi)\prod_{p|d}\left(1-\frac{\chi(p)}{p^s} \right),
 \end{align*}
 which is absolutely convergent for $\Re(s)>1/2$. Using Lemma \eqref{parronlemma}, we have
 \[\sum_{\substack{t\leq \frac{Q}{dr^2}\\(t,d)=1}}\chi(t)=\sum_{t\leq \frac{Q}{dr^2}}\chi(t)\sum_{q|(t,d)}\mu(q)=\frac{1}{2\pi i}\int_{\frac{3}{2}-iT}^{\frac{3}{2}+iT}\frac{\left(\frac{Q}{dr^2}\right)^sF(s)}{s}ds+R(T), \numberthis\label{s16}\]
 where 
 \[|R(T)|\ll \frac{Q}{T}\sum_{n=1}^\infty \frac{1}{n\left|\text{log}\frac{Q}{n}\right|}\ll \frac{Q^{\frac{3}{2}}}{(dr)^{\frac{3}{2}}T}. \]
As in \eqref{52} before, we estimate the integral in \eqref{s16} by deforming the path into a rectangle with points $1/2\pm iT,\ 3/2\pm iT$ and use Cauchy's residue theorem to obtain
 \[\sum_{\substack{t\leq \frac{Q}{dr^2}\\(t,d)=1}}\chi(t)\ll \frac{Q^{\frac{1}{2}}\log Q}{(dr)^{\frac{3}{2}}}. \]
 Inserting the above estimate into \eqref{s3}, we have
 \[ S_{11}\ll_k\frac{Q^{\frac{1}{2}}\log Q}{d^{\frac{3}{2}}}\sum_{\substack{r^2\leq \frac{Q}{d}\\(r^2,d)=1}}\frac{1}{r^\frac{3}{2}}\ll_k\frac{Q^{\frac{1}{2}}\log Q}{d^{\frac{3}{2}}}.\numberthis\label{s5}\]
 So, \eqref{s5} in conjunction with \eqref{s2} gives
 \[S_1\ll_k Q^{\frac{3}{2}}\log Q\sum_{d\leq Q}\frac{1}{d^{\frac{5}{2}}}\ll_k Q^{\frac{3}{2}}\log Q.\numberthis\label{s4}\]
%Inserting \eqref{s4} and \eqref{s6} into \eqref{s1} gives the result for $\chi\ne \chi_0.$
Next, we estimate the second sum on the right side of \eqref{s1}.  %Employing the formula $\phi(n)=\sum_{d|n}\mu(d)n/d$, and using \eqref{s5} we have
 \begin{align*}
 \iffalse
     \sum_{\substack{s\leq Q\\ \mu(s)^2=1}}\chi(s)\phi(s)&= \sum_{\substack{s\leq Q\\ \mu(s)^2=1}}\chi(s)\sum_{d|s}\frac{\mu(d)s}{d} \\
     \fi
    S_2   &=\sum_{d\leq Q}\frac{\mu(d)}{d} \sum_{\substack{s\leq Q\\ d|s \\ \mu(s)^2=1}}\chi(s)s 
      =\sum_{d\leq Q}\chi(d)\mu(d) \sum_{\substack{m\leq \frac{Q}{d} \\ (m,d)=1}}\chi(m)\mu(m)^2m 
    \\& =\sum_{d\leq Q}\chi(d)\mu(d)\sum_{j\leq \frac{Q}{d}}\sum_{\substack{j\leq m\leq \frac{Q}{d} \\ (m,d)=1}}\chi(m)\mu(m)^2 .\numberthis\label{s6}
 \end{align*}
For $\chi\ne \chi_0,$
\[S_2\ll_k Q^{\frac{1}{2}}\log Q\sum_{d\leq Q}\sum_{j\leq \frac{Q}{d}}\frac{1}{d^{\frac{3}{2}}}
     \ll_k Q^{\frac{3}{2}}\log Q.\numberthis\label{S2nonprin}\]
For $\chi=\chi_0$,
\iffalse
\begin{align*}
   S_2&=\sum_{\substack{s\leq Q\\ \mu(s)^2=1}}\chi_0(s)s\sum_{d|s}\frac{\mu(d)}{d} \\
    &=\sum_{d\leq Q}\frac{\mu(d)}{d}\sum_{\substack{s\leq Q\\ d|s}}\chi_0(s)s\mu(s)^2 \\
    &=\sum_{d\leq Q}\chi_0(d)\mu(d)\sum_{j\leq \frac{Q}{d}}\sum_{\substack{j\leq m\leq \frac{Q}{d}\\ (m,d)=1}}\chi_0(m)\mu(m)^2.\numberthis\label{s7}
\end{align*}
\fi
the inner sum in \eqref{s6} is estimated below:
\begin{align*}
  S_{21}&:=  \sum_{\substack{m\leq \frac{Q}{d}\\ (m,d)=1}}\chi_0(m)\mu(m)^2=\sum_{\substack{m\leq \frac{Q}{d}\\(m,d)=1\\(m,k)=1}}\mu(m)^2
    =\sum_{\substack{m\leq \frac{Q}{d}\\ (m,dk)=1}}\sum_{t^2|m}\mu(t)=\sum_{t^2\leq \frac{Q}{d}}\mu(t)\sum_{\substack{m\leq \frac{Q}{d}\\t^2|m\\(m,dk)=1}}1\\&=\sum_{t^2\leq \frac{Q}{d}}\mu(t)\sum_{\substack{n\leq \frac{Q}{t^2d}\\(nt^2,dk)=1}}1=\sum_{\substack{t^2\leq \frac{Q}{d}\\(t^2,dk)=1}}\mu(t)\sum_{\substack{n\leq \frac{Q}{t^2d}\\(n,dk)=1}}1\sum_{\substack{t^2\leq \frac{Q}{d}\\(t^2,dk)=1}}\mu(t)\sum_{n\leq \frac{Q}{t^2d}}\sum_{\substack{q|n\\ q|dk}}\mu(q)\\
    %&=\sum_{\substack{t^2\leq \frac{Q}{d}\\(t^2,dk)=1}}\mu(t)\sum_{q|dk}\mu(q)\sum_{\substack{n\leq \frac{Q}{t^2d}\\q|n}}1\\
   & =\sum_{\substack{t^2\leq \frac{Q}{d}\\(t^2,dk)=1}}\mu(t)\sum_{q|dk}\mu(q)\sum_{s\leq \frac{Q}{qt^2d}}1=\sum_{\substack{t^2\leq \frac{Q}{d}\\(t,dk)=1}}\mu(t)\sum_{q|dk}\mu(q)\left(\frac{Q}{qt^2d}+\BigO{1} \right)\\
    &=\frac{Q}{d}\sum_{\substack{t^2\leq \frac{Q}{d}\\(t,dk)=1}}\frac{\mu(t)}{t^2}\sum_{q|dk}\frac{\mu(q)}{q}+\BigOke{\frac{Q^{\frac{1}{2}}}{d^{\frac{1}{2}-\epsilon}}}
    %&=\frac{Q\phi(dk)}{kd^2}\sum_{\substack{t\leq \left(\frac{Q}{d}\right)^{\frac{1}{2}}\\(t,dk)=1}}\frac{\mu(t)}{t^2}+\BigOk{\frac{Q^{\frac{1}{2}}}{d^{\frac{1}{2}-\epsilon}}}\\
    =\frac{Q\phi(dk)}{kd^2\zeta(2)}\prod_{p|dk}{\left(1-\frac{1}{p^2}\right)}^{-1}+\BigOke{\frac{Q^{\frac{1}{2}}}{d^{\frac{1}{2}-\epsilon}}} .\numberthis\label{6.10}
\end{align*}
Inserting \eqref{6.10} into \eqref{s6}, we have
\[ S_2=\frac{3Q^2}{k\pi^2}\prod_{p|k}\left(1-\frac{1}{p^2}\right)^{-1}\sum_{d\leq Q}\frac{\chi_0(d)\mu(d)\phi(dk)}{d^3}\prod_{\substack{p|d\\p\nmid k}}\left(1-\frac{1}{p^2}\right)^{-1}+\BigOk{Q^{\frac{3}{2}}}.\numberthis\label{s8} \]
    The Dirichlet series for $\chi_0(d)\mu(d)\phi(dk)d^{-3}\prod_{\substack{p|d\\p\nmid k}}\left(1-\frac{1}{p^2}\right)^{-1}$ is
    \begin{align*}
        F(s)&=\sum_{n=1}^{\infty}\frac{\chi_0(n)\mu(n)\phi(nk)}{n^{s+3}}\prod_{\substack{p|n\\p\nmid k}}\left(1-\frac{1}{p^2}\right)^{-1}=\prod_{p\nmid k}\left[1-\frac{\phi(pk)}{p^{s+3}}\left(1-\frac{1}{p^2}\right)^{-1} \right]=\prod_{p\nmid k}\left(1-\frac{\phi(k)}{p^{s+1}(p+1)} \right),
    \end{align*}
and is convergent for $\Re(s)>-1$.  Using Lemma \ref{parronlemma}, we have
    \[ \sum_{n\leq Q}\frac{\chi_0(n)\mu(n)\phi(nk)}{n^{3}}\prod_{p|n}\left(1-\frac{1}{p^2}\right)^{-1}=\frac{1}{2\pi i}\int_{1-iT}^{1+iT}\frac{Q^sF(s)}{s}ds+R(T), \]
    where
    \[|R(T)|\ll \frac{Q}{T}\sum_{n=1}^\infty \frac{1}{n\left|\text{log}\frac{Q}{n}\right|}\ll \frac{Q}{T}. \]
We deform the path of integration into a contour with line segments [$1-iT$,$1+iT$], [$1-iT$,$-1/2-iT$], [$-1/2-iT$,$-1/2+iT],$ and [$-1/2+iT$,$1+iT].$ Since the integrand has a simple pole at $s=0$ inside this contour. By Cauchy's residue theorem, we have
\[\frac{1}{2\pi i}\int_{1-iT}^{1+iT}\frac{Q^sF(s)}{s}ds=\prod_{p\nmid k}\left(1-\frac{\phi(k)}{p(p+1)} \right)+\sum_{j=1}^{3}I_j, \]
where first term on right side is the residue of the integrand at the pole $s=0$,  $I_1$ and $I_3$ are the integrals along the horizontal segments $[1-i T,-1/2-i T]$ and $[-1/2+i T,1+i T],$ respectively and $I_2$ is the integral over the vertical segment $[-1/2-i T,-1/2+i T].$
\begin{align*}
    |I_1|,|I_3|\ll_k \int_{-\frac{1}{2}}^{1}\frac{Q^{\sigma}}{|\sigma+iT|}d\sigma \ll_k \frac{Q\log Q}{T},
\end{align*}
and
\[|I_2|\ll_k  Q^{-\frac{1}{2}}\int_{-T}^{T} \frac{1}{|-1/2+it|}dt\ll_k \frac{\log T}{Q^{\frac{1}{2}}}.\]
Collecting all estimates and choosing $T=Q^2$, we have
\[\sum_{n\leq Q}\frac{\chi_0(n)\mu(n)\phi(nk)}{n^{3}}\prod_{p|n}\left(1-\frac{1}{p^2}\right)^{-1} =\prod_{p\nmid k}\left(1-\frac{\phi(k)}{p(p+1)} \right)+\BigOk{\frac{\log Q}{Q^{\frac{1}{2}}}}.\]
This together with \eqref{s8} gives
\[ S_2=\frac{Q^2}{2k L(2,\chi_0)}\prod_{p\nmid k}\left(1-\frac{\phi(k)}{p(p+1)} \right)+\BigOk{Q^{\frac{3}{2}}}.\numberthis\label{s13}\]
    Inserting this and \eqref{s21} into \eqref{s1} provides the result for $\chi=\chi_0$. For $\chi\ne\chi_0$, we obtain the required result by inserting \eqref{s4} and \eqref{S2nonprin} in  \eqref{s1}.
    \end{proof}
\begin{thm}\label{t3}
  Fix a positive integer $k$, and a Dirichlet character $\chi$ modulo k. Then, for all large positive integers $Q,$ we have 
 \begin{align*}
        \sum_{\substack{s\leq Q\\ \mu(s)^2=1}}\chi(s)\delta(s)&=\left\{\begin{array}{cc}
         \BigOk
{Q^{\frac{3}{2}}(\log Q)},      &  \mbox{if}\ \chi\ne\chi_0\\ \\ 
\dfrac{12Q^2}{\pi^2}\prod_{p|k}\dfrac{p^2}{p^2+p-1}\prod_{p}\dfrac{p^3+p^2-p}{p^3+p^2-p-1}\\
       -\dfrac{12Q^2}{\pi^2}\prod_{p|k}\dfrac{p^2}{p^2+p-1}\prod_{p}\left(1-\dfrac{1}{p(p+1)} \right)\\-\dfrac{Q^2}{k L(2,\chi_0)}\prod_{p\nmid k}\left(1-\dfrac{\phi(k)}{p(p+1)} \right)+\BigOk
{Q^{\frac{3}{2}}(\log Q)^{\frac{5}{2}}},      &  \mbox{if}\ \chi=\chi_0.
       \end{array}\right.
       \end{align*}
 \end{thm}   
 \begin{proof}
 For $\chi\ne \chi_0$, we have (see \cite{MR2424917})
 \[ \sum_{\substack{s\leq Q\\ \mu(s)^2=1}}\chi(s)\delta(s)\ll_k \BigOk
{Q^{\frac{3}{2}}(\log Q)}. \]
 For $\chi=\chi_0$,   \eqref{deficiency} gives 
\begin{align*}
    \sum_{\substack{s\leq Q\\ \mu(s)^2=1}}\chi_0(s)\delta(s)&=\sum_{\substack{s\leq Q\\ \mu(s)^2=1}}\chi_0(s)\phi(s)\left\lfloor{\frac{2Q}{s}}\right\rfloor-\sum_{\substack{s\leq Q\\ \mu(s)^2=1}}\chi_0(s)T(s)\\
    &=\sum_{\substack{s\leq 2Q\\ \mu(s)^2=1}}\chi_0(s)\phi(s)\left\lfloor{\frac{2Q}{s}}\right\rfloor-\sum_{\substack{Q<s\leq 2Q\\ \mu(s)^2=1}}\chi_0(s)\phi(s)-\sum_{\substack{s\leq Q\\ \mu(s)^2=1}}\chi_0(s)T(s). \numberthis\label{s14}
    \end{align*}
Using \eqref{s13}, we have
 \[ \sum_{\substack{Q<s\leq 2Q\\ \mu(s)^2=1}}\chi_0(s)\phi(s)=\frac{3Q^2}{2k L(2,\chi_0)}\prod_{p\nmid k}\left(1-\dfrac{\phi(k)}{p(p+1)} \right)+\BigOk{Q^\frac{3}{2}}. \numberthis\label{6.13}\]
 \iffalse
 Inserting the above estimate into \eqref{s14} and using Theorem \ref{t1}, we have
 \begin{align*}
  \sum_{\substack{s\leq Q\\ \mu(s)^2=1}}\chi_0(s)\delta(s)=&\sum_{\substack{s\leq 2Q\\ \mu(s)^2=1}}\chi_0(s)\phi(s)\left\lfloor{\frac{2Q}{s}}\right\rfloor-\dfrac{12Q^2}{\pi^2}\prod_{p|k}\dfrac{p^2}{p^2+p-1}\prod_{p}\left(1-\dfrac{1}{p(p+1)} \right)\\&-\dfrac{Q^2}{k L(2,\chi_0)}\prod_{p\not|\;k}\left(1-\dfrac{\phi(k)}{p(p+1)} \right)+\BigOk{Q^\frac{3}{2}}. \numberthis\label{s15}
  \end{align*}
  \fi
Moreover \cite[see (3.13)]{MR2424917},
 \[\sum_{\substack{s\leq Q\\ \mu(s)^2=1}}\chi_0(s)\phi(s)\left\lfloor{\frac{2Q}{s}}\right\rfloor=\frac{12Q^2}{\pi^2}\prod_{p|k}\dfrac{p^2}{p^2+p-1}\prod_{p}\dfrac{p^3+p^2-p}{p^3+p^2-p-1}+\BigOk{Q^{\frac{3}{2}}(\log Q)^{\frac{5}{2}}}.\numberthis\label{6.14} \]
 Inserting \eqref{6.13} and \eqref{6.14}, and using Theorem \ref{t1} in \eqref{s14}, we get the required result for $\chi=\chi_0$.
 \end{proof}
 
 \begin{thm}\label{t2}
Fix a positive integer $k$, and a Dirichlet character $\chi$ modulo k. Then, for all large positive integers $Q,$ we have 
  \begin{align*}
        \mathcal{M}_{2,\Box}(\chi,Q)&=\left\{\begin{array}{cc}
       4Q^2L(1,\chi)\prod_p\left(1-\dfrac{\chi(p^2)}{p^2}\right)\left(1-\dfrac{\chi(p)}{p^2}\left(1
       +\dfrac{\chi(p)}{p}\right)^{-1}\right) 
        \\ \BigOk
{Q^{\frac{3}{2}}(\log Q)},      &  \mbox{if}\ \chi\ne\chi_0\\ \\ 
\dfrac{24Q^2}{\pi^2}\left(\log 2Q+2\gamma-\dfrac{3}{2}-\dfrac{2\zeta^{\prime}(2)}{\zeta(2)}+\sum_{p|k}\dfrac{\log p}{p+1}\right.\\\left. +\sum_{p\nmid k}\dfrac{\log p}{(p+1)(p^2+p-1)} \right)\prod_{p|k}\dfrac{p}{p+1}\prod_{p\nmid k}\left(1-\dfrac{1}{p(p+1)} \right) \\
       -\dfrac{12Q^2}{\pi^2}\prod_{p|k}\dfrac{p^2}{p^2+p-1}\prod_{p}\left(1-\dfrac{1}{p(p+1)} \right)\\-\dfrac{Q^2}{2k L(2,\chi_0)}\prod_{p\nmid k}\left(1-\dfrac{\phi(k)}{p(p+1)} \right)+\BigOk
{Q^{\frac{3}{2}}(\log Q)^{\frac{5}{2}}},      &  \mbox{if}\ \chi=\chi_0.
       \end{array}\right.
       \end{align*}
\end{thm}

\begin{proof}
    We  follow  Theorem \ref{thm4} with $\mu(s)^2$ and obtain 
  \begin{align*}
     \mathcal{M}_{2,\Box}(\chi,Q)=&2\sum_{n\leq 2Q}(2Q-n)h_{\chi}(n)
       -\sum_{\substack{s\leq 2Q \\ \mu(s)^2=1}}\chi(s)\phi(s){\left \lfloor {\frac{2Q}{s}}\right \rfloor}\\ &+3\sum_{\substack{Q<s\leq 2Q \\ \mu(s)^2=1}}\chi(s)\phi(s)-4Q\sum_{\substack{Q<s\leq Q \\ \mu(s)^2=1}}\frac{\chi(s)\phi(s)}{s}+\sum_{\substack{s\leq Q\\\mu(s)^2=1}}\chi(s)\delta(s)+\BigO{Q(\log Q)^2}\\
      &= M_1+M_2+M_3+M_4+M_5+\BigO{Q(\log Q)^2}, \numberthis\label{s25} 
 \end{align*}
where 
 \[h_{\chi}(n)=\sum_{s|n}\frac{\chi(s)\mu(s)^2\phi(s)}{s} .\]
 One can obtain $M_5$ from Theorem \ref{t3}. Estimation of $M_2$ is in \eqref{6.14} principal characters and for $\chi\ne \chi_0$: from \cite{MR2424917}, we have \[M_2\ll Q^{3/2}\log Q\numberthis\label{sfm2nonprin}.\] From \eqref{6.13} we obtain $M_3$ for for $\chi=\chi_0$. From \eqref{S2nonprin}, one can obtain for $\chi\ne\chi_0$
 \[M_3\ll Q^{3/2}\log Q.\numberthis\label{sfm3nonprin}\]
 Case (i): For  $\chi\ne \chi_0$
  \begin{align*}
    M_4&=\sum_{\substack{Q<s\leq 2Q\\ \mu(s)^2=1}}\chi(s)\sum_{d|s}\frac{\mu(d)}{d}
    =\sum_{d\leq 2Q}\frac{\mu(d)}{d} \sum_{\substack{Q<s\leq 2Q\\ d|s}}\chi(s)\mu(s)^2\\
    &=\sum_{d\leq 2Q}\frac{\chi(d)\mu(d)}{d} \sum_{\substack{Q<m\leq 2Q\\ (m,d)=1}}\chi(s)\mu(s)^2 \ll_k Q^{\frac{1}{2}}.\numberthis\label{6.23}
\end{align*}
Case (ii): For $\chi=\chi_0$, using \eqref{6.23} and \eqref{6.10}, we have
\begin{align*}
   M_4
    &=\frac{Q}{k L(2,\chi_0)}\prod_{p\nmid k}\left(1-\frac{\phi(k)}{p(p+1)} \right)+\BigOk{Q^{\frac{1}{2}}}.\numberthis\label{sfm5prin}
\end{align*}
 Using \eqref{int}, we have
 \[M_1=\frac{1}{2\pi i}\int_{2-i\infty}^{2+i\infty}\frac{(2Q)^{s+1}H_{\chi}(s)}{s(s+1)}ds, \numberthis\label{s26} \]
 where %Now, we estimate the first sum in \eqref{s25}. To do so we write the Dirichlet series for $h_{\chi}(n)$ which is given by
\begin{align*}
    H_{\chi}(s)&=\sum_{n=1}^{\infty}\frac{h_{\chi}(n)}{n^s}=\sum_{n=1}^{\infty}\sum_{m|n}\frac{\chi(m)\phi(m)\mu(m)^2}{mn^s}=\sum_{m=1}^{\infty}\frac{\chi(m)\phi(m)\mu(m)^2}{m^{s+1}}\sum_{k=1}^{\infty}\frac{1}{k^s}\\
    &=\zeta(s)\sum_{n=1}^{\infty}\frac{\chi(m)\mu(m)^2}{m^{s+1}}\sum_{d|m}\frac{\mu(d)}{d}=\zeta(s)\sum_{d=1}^{\infty}\frac{\mu(d)}{d}\sum_{\substack{m=1\\d|m}}^{\infty}\frac{\chi(m)\mu(m)^2}{m^s}\\
      &=\zeta(s)\sum_{d=1}^{\infty}\frac{\chi(d)\mu(d)}{d^{s+1}}\sum_{\substack{v=1\\(v,d)=1}}^{\infty}\frac{\chi(v)\mu(v)^2}{v^s}=\zeta(s)\sum_{d=1}^{\infty}\frac{\chi(d)\mu(d)}{d^{s+1}}\prod_{p\nmid d}\left(1+\frac{\chi(p)}{p^s} \right)\\
      &=\zeta(s)L(s,\chi)\prod_{p}\left(1-\frac{\chi(p^2)}{p^{2s}} \right)\sum_{d=1}^{\infty}\frac{\chi(d)\mu(d)}{d^{s+1}}\prod_{p|d}\left(1+\frac{\chi(p)}{p^s} \right)^{-1} \\
       &=\zeta(s)L(s,\chi)\prod_{p}\left(1-\frac{\chi(p^2)}{p^{2s}} \right)\prod_{p}\left(1-\frac{\chi(p)}{p^{s+1}}\left(1+\frac{\chi(p)}{p^s} \right)^{-1} \right)\\
       &=\zeta(s)L(s,\chi)\prod_{p}\left(1-\frac{\chi(p^2)}{p^{2s}} \right)\left(1-\frac{\chi(p)}{p^{s+1}}\left(1+\frac{\chi(p)}{p^s} \right)^{-1} \right). \numberthis\label{Hs}
\end{align*}
 which is absolutely convergent for $\Re(s)>1$. \\
Case (i): For $\chi\ne\chi_0$:  We deform the path of integration into a contour with line segments $(2-i\infty, 2-iT],\ [2-iT, 1/2+\eta-iT],\  [1/2+\eta-iT, 1/2+\eta+iT],\ [1/2+\eta+iT, 2+iT]\ \mbox{and}\ [2+iT, 2+i\infty)$. By Cauchy's residue theorem, we have 
 \begin{align*}
 \frac{1}{2\pi i}\int_{2-i\infty}^{2+i\infty}\frac{(2Q)^{s+1}H_{\chi}(s)}{s(s+1)}ds=& 2Q^2L(1,\chi)\prod_p\left(1-\dfrac{\chi(p^2)}{p^2}\right)\left(1-\dfrac{\chi(p)}{p^2}\left(1
       +\dfrac{\chi(p)}{p}\right)^{-1}\right)+\sum_{j=1}^{5}I_j,
 \end{align*}
 where first term on right side is the residue at the pole $s=1$, $I_1$, $I_3$ and $I_5$ are the integrals along the vertical segments $(2-i\infty, 2-iT]$, $[1/2+\eta-iT, 1/2+\eta+iT]$, and $[2+iT, 2+i\infty)$ respectively and $I_2$ and $I_4$ are the integrals over the horizontal segments $[2-iT, 1/2+\eta-iT]$ and $[1/2+\eta+iT, 2+iT]$, respectively. We estimate the integrals using the same technique as in \cite{MR2424917} 
 \[|I_1|, |I_5|\ll_k \frac{Q^3}{T}; \ |I_2|, |I_4|\ll_k Q^3T^{-1+2\epsilon}; \ |I_3|\ll_k Q^{\frac{3}{2}}\log T. \]
Choosing $T=Q^2$ and collecting all estimate in \eqref{s26}, we have
 \[M_1=2Q^2L(1,\chi)\prod_p\left(1-\dfrac{\chi(p^2)}{p^2}\right)\left(1-\dfrac{\chi(p)}{p^2}\left(1
       +\dfrac{\chi(p)}{p}\right)^{-1}\right)+\BigOkb{Q^{\frac{3}{2}}\log Q}.\numberthis\label{sfm1nonprin}\]
       \iffalse
 Inserting the above estimate into \eqref{s25} and using \eqref{s6}, we have  \begin{align*}
     \mathcal{M}_{2,\Box}(\chi,Q)=2Q^2L(1,\chi)\prod_p\left(1-\dfrac{\chi(p^2)}{p^2}\right)\left(1-\dfrac{\chi(p)}{p^2}\left(1
       +\dfrac{\chi(p)}{p}\right)^{-1}\right)
 \end{align*}
 
\fi
Case (ii): For $\chi=\chi_0$, \eqref{Hs} gives
\[H_{\chi_0}(s)=\frac{\zeta(s)^2}{\zeta(2s)}\prod_{p|k}{\left(1+\frac{1}{p^s} \right)}^{-1}\prod_{p\nmid k}\left(1-\frac{1}{p(p^s+1)} \right).\]
We deform the path of integration in \eqref{s26} into a contour with line segments $(2-i\infty, 2-iT],\ [2-iT, 1/2-iT],\  [1/2-iT, 1/2+iT],\ [1/2+iT, 2+iT]\ \mbox{and}\ [2+iT, 2+i\infty)$. By Cauchy's residue theorem, we have 
 \begin{align*}
 \frac{1}{2\pi i}\int_{2-i\infty}^{2+i\infty}\frac{(2Q)^{s+1}H_{\chi_0}(s)}{s(s+1)}ds=& \sum_{j=1}^{5}I_j+\text{Res}_{s=1}\left(\frac{(2Q)^{s+1}H_{\chi_0}(s)}{s(s+1)}\right),
\end{align*}  
where  $I_1$, $I_3$ and $I_5$ are the integrals along the vertical segments $(2-i\infty, 2-iT]$, $[1/2-iT, 1/2+iT]$, and $[2+iT, 2+i\infty)$ respectively and $I_2$ and $I_4$ are the integrals over the horizontal segments $[2-iT, 1/2-iT]$ and $[1/2+iT, 2+iT]$, respectively and second term on the right side is the residue by the double pole at $s=1$ which is given by
\begin{align*}
&\frac{24Q^2}{k\pi^2}\left(\log 2Q+2\gamma-\dfrac{3}{2}-\dfrac{2\zeta^{\prime}(2)}{\zeta(2)}+\sum_{p|k}\dfrac{\log p}{p+1} +\sum_{p\nmid k}\dfrac{\log p}{(p+1)(p^2+p-1)} \right)\\ &\times\prod_{p|k}\dfrac{p}{p+1}\prod_{p\nmid k}\left(1-\dfrac{1}{p(p+1)} \right). 
\end{align*}
We estimate the integrals as in \cite{MR2424917}
\[|I_1|, |I_5|\ll_k\frac{Q^3}{T}; \ |I_2|, |I_4|\ll_k Q^3T^{-1+\epsilon}; \ |I_3|\ll_k Q^{\frac{3}{2}}(\log T)^2. \]
 Choosing $T=Q^3$ and collecting all estimates in \eqref{s26}, we have
 \begin{align*}
 M_1=&\frac{24Q^2}{k\pi^2}\left(\log 2Q+2\gamma-\dfrac{3}{2}-\dfrac{2\zeta^{\prime}(2)}{\zeta(2)}+\sum_{p|k}\dfrac{\log p}{p+1} +\sum_{p\nmid k}\dfrac{\log p}{(p+1)(p^2+p-1)} \right)\\ &\times\prod_{p|k}\dfrac{p}{p+1}\prod_{p\nmid k}\left(1-\dfrac{1}{p(p+1)} \right)+\BigOk{Q^{\frac{3}{2}}(\log Q)^2}. \numberthis\label{sfm1prin}
\end{align*}
For $\chi=\chi_0$, we obtain required result by inserting Theorem \ref{t3}, \eqref{6.14}, \eqref{6.13}, \eqref{sfm5prin}, \eqref{sfm1prin} in \eqref{s25}. Inserting Theorem \ref{t3}, \eqref{sfm2nonprin}, \eqref{sfm3nonprin}, \eqref{6.23}, \eqref{sfm1nonprin} in \eqref{s25} gives the required result for $\chi\ne\chi_0$.
\iffalse
Inserting the above estimate into \eqref{s25}, and using \eqref{6.14} and \eqref{6.13}, we have for $\chi=\chi_0$
\begin{align*}
    \mathcal{M}_{2,\Box}(\chi_0,Q)=&\frac{48Q^2}{k\pi^2}\left(\log 2Q+2\gamma-\dfrac{3}{2}-\dfrac{2\zeta^{\prime}(2)}{\zeta(2)}+\sum_{p|k}\dfrac{\log p}{p+1} +\sum_{p\not|\;k}\dfrac{\log p}{(p+1)(p^2+p-1)} \right)\\ &\times\prod_{p|k}\dfrac{p}{p+1}\prod_{p\not|\;k}\left(1-\dfrac{1}{p(p+1)} \right) -\frac{24Q^2}{\pi^2}\prod_{p|k}\dfrac{p^2}{p^2+p-1}\prod_{p}\dfrac{p^3+p^2-p}{p^3+p^2-p-1}\\ &+\frac{9Q^2}{2k L(2,\chi_0)}\prod_{p\not|\;k}\left(1-\dfrac{\phi(k)}{p(p+1)} \right)-4Q\sum_{\substack{Q<s\leq Q \\ \mu_{\mathcal{B}}(s)=1}}\frac{\chi_0(s)\phi(s)}{s}+\BigOkb{Q^{\frac{3}{2}}(\log Q)^{\frac{5}{2}}}.\numberthis\label{6.22}
\end{align*}

Inserting the above estimate into \eqref{6.22} gives the required result for $\chi=\chi_0$.
\fi
\end{proof}

\bibliographystyle{amsalpha}
\bibliography{reference} 
    \end{document}